\documentclass[11pt]{article}

\usepackage{amsmath,amssymb,amsthm}
\usepackage[margin=1in]{geometry}

\newtheorem{theorem}{Theorem}[section]
\newtheorem{proposition}[theorem]{Proposition}
\newtheorem{lemma}[theorem]{Lemma}
\newtheorem{corollary}[theorem]{Corollary}

\theoremstyle{definition}
\newtheorem{definition}[theorem]{Definition}
\newtheorem{example}[theorem]{Example}
\theoremstyle{remark}
\newtheorem{remark}[theorem]{Remark}

\newcommand{\QQ}{\mathbb Q}
\newcommand{\La}{\Lambda}
\newcommand{\A}{\mathcal A}
\newcommand{\R}{\mathcal R}
\newcommand{\Q}{\mathcal Q}
\newcommand{\D}{\mathcal D}
\newcommand{\Lc}{\mathcal L}
\newcommand{\LL}{\mathbf{L}}
\newcommand{\Ps}{\Psi}
\newcommand{\Id}{\operatorname{Id}}
\newcommand{\bT}{\mathbf{T}}
\newcommand{\downmap}{z\,\pi_-}
\newcommand{\Mq}{\mathsf q}
\newcommand{\Mt}{\mathsf t}
\newcommand{\ip}[2]{\left\langle #1,#2\right\rangle}

\title{On the Nazarov--Sklyanin Lax operator for Macdonald functions}
\author{Ryan Mickler}
\date{\today}

\begin{document}
\maketitle
\begin{abstract}
We study the Lax operator for Macdonald symmetric functions introduced by
Nazarov and Sklyanin in 2015, extending the spectral theory developed by Mickler--Moll in the
Jack case.  We extend the NS operator to a Lax operator acting on symmetric functions with coefficients in polynomials in an auxiliary variable,
and prove it has an orthogonal eigenbasis indexed
by partitions and their addable corners, with the corner contents as
eigenvalues, and leading terms given by Macdonald symmetric functions.  We obtain the Macdonald analogue of the author's earlier Jack box-sum identity
for Littlewood--Richardson coefficients.  We then strengthen this box-sum
identity to a two-parameter formula, proved using a finite-rank product
formula of Warnaar.  Its residues give explicit linear relations among
Macdonald Littlewood--Richardson coefficients.
\end{abstract}

\section{Introduction}\label{sec:intro}

Let $\La$ be the ring of symmetric functions in variables $x_i$.  Nazarov and
Sklyanin~\cite{Nazarov:2015ab} introduced a Lax operator $\Lc^{NS}$ on
$\La^\infty :=z^{-1}\La[z^{-1}]$ whose resolvent, projected to $\La$, generates a commuting family of operators
that acts diagonally on the Macdonald symmetric functions $J_\lambda$.  This is the
$\Mq,\Mt$ analogue of their Lax construction for Jack symmetric functions
\cite{Nazarov:2013}.  In joint work with Moll \cite{MicklerMoll:2023}, the
author developed the corresponding Jack spectral theory, which was then used
to prove Littlewood--Richardson identities by Lax-operator methods
\cite{Mickler:2023}.

In \eqref{eq:Laxdef} we construct a formally self-adjoint Lax operator
$\LL$ acting on $\La[z^{-1}]$, with a $\Mq$-dependent inner product.  We
use two projections: $\pi_0$ sends $z^{-1}$ to zero, while
$\pi_\infty$ sets all variables $x_i$ to zero.

\begin{theorem}[Spectral theorem; Propositions~\ref{prop:spectral},
\ref{prop:jackpsiexp} and~\ref{prop:orthogonal}]
There is an orthogonal basis of $\La[z^{-1}]$ indexed by pairs $(\lambda,s)$,
where $\lambda$ is a partition and $s\in\A_\lambda$ is an addable corner.
Its vectors $\Ps_\lambda^s$ satisfy
\[
\LL\Ps_\lambda^s=[s]\Ps_\lambda^s,
\qquad
\pi_0\Ps_\lambda^s=J_\lambda,
\]
where $[s]=\Mt^i\Mq^{-j}$ is the content of the (zero-based) corner $s=(i,j)$.
Furthermore, the integral Macdonald function decomposes as
\begin{equation}
J_\lambda = \sum_{s \in \A_\lambda} \tau_\lambda^s \Ps_\lambda^s,
\end{equation}
where $\tau_\lambda^s$ are the Kerov transition measures
(Definition~\ref{def:kerov}).
\end{theorem}

As a consequence, we show that the modified
Nazarov--Sklyanin Lax operator $\LL^+=\D_\Mq^{-1}\Lc^{NS}$ on $\La^\infty$ has an
orthogonal eigenbasis indexed by partitions and their removable corners,
with the outer-corner contents as eigenvalues
(Corollary~\ref{cor:NSspectral}).  

We demonstrate several more properties of the full Lax operator $\LL$. Define the single-box factor
\[
\mathcal N(u)=
\frac{(1-u^{-1}[1,1])(1-u^{-1}[0,0])}
     {(1-u^{-1}[0,1])(1-u^{-1}[1,0])}
\]
and, for partitions $\mu,\nu$, the \emph{box-star factor}
\[
\bT_{\mu\star\nu}(u):=\prod_{x\in\mu,\,y\in\nu}\mathcal N(u/[x+y]),
\]
where $x+y$ is the coordinatewise sum of the zero-based boxes $x$ and $y$, so
that $[x+y]=[x][y]$.

\begin{theorem}[Product-resolvent formula; compare {\cite[Theorem~4.9]{Mickler:2023}}]\label{thm:main}
For all partitions $\mu,\nu$, addable corners $s\in\A_\mu$, $t\in\A_\nu$, the resolvent of the Macdonald Lax operator satisfies
\begin{equation}\label{eq:maintheorem}
\pi_\infty(u-\LL)^{-1}\bigl(\Ps_\mu^s\Ps_\nu^t\bigr)
=
\frac{\bT_{\mu\star\nu}(u)}{u-[s+t]}\;  \pi_\infty\bigl(\Ps_\mu^s\Ps_\nu^t\bigr).
\end{equation}
\end{theorem}

After normalizing by $\pi_\infty(\Ps_\mu^s\Ps_\nu^t)$, the dependence on the
corners $s,t$ is confined to the factor $(u-[s+t])^{-1}$.  This factor need
not give a genuine pole, as it may cancel against a zero of
$\bT_{\mu\star\nu}$ (Remark~\ref{rem:corner-cancellation}).  The proof, by
induction on the size of the smaller partition, is given in
Section~\ref{sec:proof}.

In the Jack case, \cite{Mickler:2023} proves a box-sum identity for
Littlewood--Richardson coefficients by computing three traces of the
derivation defect $\mathcal L(fg)-\mathcal L(f)g-f\mathcal L(g)$ of the Jack
Lax operator $\mathcal L$.  Here the corresponding object is the normalized
product defect $\beta(f,g)=fg-\LL^{-1}(\LL(f)\LL(g))$
(Definition~\ref{def:beta}).  Its three traces together yield the Macdonald
box-sum identity (Theorem~\ref{thm:boxsummain}), generalizing the Jack case.

We then strengthen the box-sum identity to a two-parameter identity.  Let
$\epsilon_v:\La\to\QQ(\Mq,\Mt,v)$,
$\epsilon_v(p_r)=(1-v^r)/(1-\Mt^r)$, be Macdonald's principal specialization
\cite[Chapter~VI, (6.16)]{Macdonald:1995}.  Define
$\kappa_{u,v}:\La\to\QQ(\Mq,\Mt,u,v)$ on the integral Macdonald basis by
\[
\kappa_{u,v}(J_\lambda)
:=\frac{\epsilon_u(J_\lambda)\,\epsilon_v(J_\lambda)}{\epsilon_{uv}(J_\lambda)},
\]
and extend $\kappa$ linearly to $\La$.
\begin{theorem}[Two-parameter box-sum identity]\label{thm:maintwo}
For all partitions $\mu,\nu$,
\begin{equation}\label{eq:kappa-multiplicativity}
\kappa_{u,v}(J_\mu J_\nu)
=\kappa_{u,v}(J_\mu)\kappa_{u,v}(J_\nu)
\bT_{\mu\star\nu}(uv).
\end{equation}
\end{theorem}
The residues of the full
identity give weighted linear relations among Macdonald
Littlewood--Richardson coefficients
(Corollaries~\ref{cor:kappa-double-residue} and~\ref{cor:kappa-root-residue}).

In Section~\ref{ssec:AFLTproof} we show that Theorem~\ref{thm:maintwo} is
equivalent to a finite-rank product formula of Warnaar
\cite[Proposition~3.1]{Warnaar:2005}.  A Lax-operator proof of
Theorem~\ref{thm:maintwo} remains open.  We also
relate the limit $v\to1$ to the normalization of the Lax eigenvectors and
express $\bT_{\mu\star\nu}(uv)$ as Macdonald's Cauchy kernel evaluated at the
alphabets of $\mu$ and $\nu$ scaled by $u$ and by $v$
(Lemma~\ref{lem:TasExp} and the paragraph following it).  This kernel form
describes the parameter dependence of the factor, while the two-parameter
identity itself is proved from Warnaar's formula.

\medskip\noindent
\textbf{Outline.}  Section~\ref{sec:background} fixes notation for symmetric
functions and partitions and reviews the Nazarov--Sklyanin construction.
Section~\ref{sec:reformulation} introduces the full Lax operator and proves its
spectral theorem.  Section~\ref{sec:ladder} computes the principal
specialization of its eigenvectors and proves the creation and down-map
formulas.  Section~\ref{sec:proof} introduces the normalized product defect
and proves its $\mathbf{y}$-trace formula, hence Theorem~\ref{thm:main}.
Section~\ref{sec:boxsum} completes the three-trace formula and deduces the
box-sum identity.  Section~\ref{sec:twoparam} introduces the two-parameter
strengthening, proves Theorem~\ref{thm:maintwo}, recovers the box-sum identity
as a limit, and derives the residue identities.

\section{Background}\label{sec:background}

\subsection{Symmetric functions and partitions}\label{ssec:notation}

Let $\La=\QQ(\Mq,\Mt)[p_1,p_2,\ldots]$ be the ring of symmetric functions with coefficients in rational functions of $\Mq,\Mt$, which we express in terms of the power sum symmetric functions $p_k = \sum_{i} x_i^k$, and for a partition $\mu$ we let $p_\mu = \prod_i p_{\mu_i}$.
The $\Mq,\Mt$-scalar product is
\[
\ip{p_\lambda}{p_\mu}_{\Mq,\Mt}
=
z_\lambda\delta_{\lambda\mu}
\prod_i\frac{1-\Mq^{\lambda_i}}{1-\Mt^{\lambda_i}},
\]
where $z_\lambda=\prod_{i\ge1}i^{m_i}m_i!$ and $m_i$ is the number of parts of
$\lambda$ equal to $i$.

We identify a partition $\lambda$ with its Young diagram, the set of boxes
$(i,j)$ with $0\le j<\lambda_{i+1}$, where rows and columns are indexed from $0$.
For a box $b=(i,j)$ of $\lambda$ we write
\[
a_\lambda(b)=\lambda_{i+1}-j-1,\qquad
\ell_\lambda(b)=\lambda'_{j+1}-i-1,\qquad
a'(b)=j,\qquad
\ell'(b)=i
\]
for its arm, leg, coarm and coleg; the coarm $a'(b)=j$ and coleg
$\ell'(b)=i$ are used for every box $b=(i,j)$, in particular for addable
corners.  We put
$n(\lambda)=\sum_{b\in\lambda}\ell'(b)=\sum_i(i-1)\lambda_i$.  The hook factors
are
\[
h_\lambda(b)=1-\Mq^{a_\lambda(b)}\Mt^{\ell_\lambda(b)+1},
\qquad
h'_\lambda(b)=1-\Mq^{a_\lambda(b)+1}\Mt^{\ell_\lambda(b)}.
\]
Let $P_\lambda=P_\lambda(\Mq,\Mt)$ be the Macdonald polynomial
\cite[Chapter~VI]{Macdonald:1995}, and let
\[
J_\lambda=\prod_{b\in\lambda}h_\lambda(b)\,P_\lambda
\]
be its integral form \cite[Chapter~VI, (8.3)]{Macdonald:1995}.  The $J_\lambda$
are orthogonal for $\ip{\cdot}{\cdot}_{\Mq,\Mt}$, with
$\|J_\lambda\|_{\Mq,\Mt}^2=\prod_{b\in\lambda}h_\lambda(b)h'_\lambda(b)$
\cite[Chapter~VI, (6.19) and (8.3)]{Macdonald:1995}.

A box $s\notin\lambda$ is an \emph{addable corner} of $\lambda$ if
$\lambda\cup\{s\}$ is a partition, and a box $r\in\lambda$ is a
\emph{removable corner} if $\lambda\setminus\{r\}$ is a partition.  We write
$\A_\lambda$ and $\R_\lambda$ for these sets, $\lambda+s$ and $\lambda-r$ for
the resulting partitions, and $b^+:=b+(1,1)$ for any box $b$.  The set of
\emph{outer corners} is
\[
\R_\lambda^+=\{r^+:r\in\R_\lambda\},
\]
and $|\A_\lambda|=|\R_\lambda|+1$.
Explicitly, let the distinct nonzero parts of $\lambda$ be
$L_1>L_2>\cdots>L_m$, with multiplicities $d_1,\ldots,d_m$, and put
$I_k=d_1+\cdots+d_k$, $I_0=0$ and $L_{m+1}=0$.  Then
\[
\A_\lambda=\{(I_k,L_{k+1}):0\le k\le m\},
\qquad
\R^+_\lambda=\{(I_k,L_k):1\le k\le m\}.
\]
We call these the \emph{corner coordinates} of $\lambda$.
The \emph{content} of a box $b=(i,j)$, whether or not it lies in $\lambda$, is
\[
[b]=[i,j]:=\Mt^i\Mq^{-j},
\]
so that $[a+b]=[a][b]$ and $[a-b]=[a]/[b]$ for the coordinatewise sum and
difference of boxes.  Since $\Mq,\Mt$ are
indeterminates, distinct boxes have distinct contents.

\subsection{The Nazarov--Sklyanin Lax operator}\label{ssec:NSlax}

We use notation from Nazarov--Sklyanin~\cite{Nazarov:2015ab} where
possible, and $P_\lambda$ and $J_\lambda$ denote Macdonald
polynomials\footnote{In \cite{Nazarov:2015ab}, $P_\lambda$ and $Q_\lambda$ are
Hall--Littlewood functions and the Macdonald functions are written
$M_\lambda$. For the Jack case, Mickler--Moll \cite{MicklerMoll:2023}
use the power series variable $w:=z^{-1}$.}.
Define the generating series
\[
Q(z)=\sum_{n\geq0}Q_nz^n
=
\exp\left(\sum_{n\geq1}\frac{1-\Mt^n}{n}p_nz^n\right),
\]
and let \(Q_n^*\) be the adjoint to multiplication by \(Q_n\), that is
\[
Q^*(z)=\sum_{n\geq0}Q_n^*z^n
=
\exp\left(\sum_{n\geq1}\frac{1-\Mq^n}{n}p_n^\perp z^n\right),
\qquad
p_n^\perp=n\frac{\partial}{\partial p_n}.
\]

Let $\La^\infty=z^{-1}\La[z^{-1}]$, so $\La[z^{-1}] = \La \oplus \La^\infty$.
For a Laurent series $f$ in $z$ with coefficients in $\La$, let $\pi_-f$, $\pi_0f$ and $\pi_{\le0}f$ be
its parts with negative, zero and nonpositive powers of $z$.  Define the
multiplication operator
\[
\D_x(z^{-m}f)=x^m z^{-m}f.
\]

\begin{definition}[Nazarov--Sklyanin Macdonald Lax operator~\cite{Nazarov:2015ab}] \label{def:nsorigoperator}
\[
\Lc^{NS}=Q^*(z^{-1}) \D_\Mt \pi_- Q(z) \in \operatorname{End}(\La^{\infty}).
\]
\end{definition}
Nazarov and Sklyanin also use the two maps
\[
\Q:\La^\infty\to\La,\qquad
\Q f=\pi_0 Q(z)f(z),
\]
and
\[
\Q^*:\La\to\La^\infty,\qquad
\Q^*g=(Q^*(z^{-1})-1)g = \pi_- Q^*(z^{-1}) g,
\]
and the scalar resolvent
\begin{equation}\label{def:resolvent}
I(u):=\Q(u\D_\Mq-\Lc^{NS})^{-1}\Q^*.
\end{equation}

In \cite{Nazarov:2014}, Nazarov and Sklyanin introduced `Macdonald' operators
\(A^{(k)}\) on \(\La\), \(k=0,1,2,\ldots\), via the series
\[
A(u)=\sum_{k\geq0}\frac{A^{(k)}}{(u;\Mt^{-1})_k},
\qquad
(u;\Mt^{-1})_k=\prod_{j=0}^{k-1}(1-u\Mt^{-j}),
\]
which is diagonal on Macdonald polynomials:
\begin{equation}\label{eq:Aeigen}
A(u)P_\lambda=A_\lambda(u)P_\lambda,
\qquad
A_\lambda(u)=
\prod_{i\geq1}
\frac{\Mq^{-\lambda_i}-u\Mt^{1-i}}{1-u\Mt^{1-i}};
\end{equation}
see also \cite[(1.19)--(1.20)]{Nazarov:2015ab}.
\begin{theorem}[Nazarov--Sklyanin, {\cite[Theorem~1]{Nazarov:2015ab}}]\label{thm:nsmain}
The scalar resolvent \eqref{def:resolvent} satisfies
\[
\frac{uI(u)}{u-1}
=
1-\frac{A(u)}{A(u\Mt^{-1})}.
\]
\end{theorem}

The eigenvalues $A_\lambda(u)$ have a product form over the boxes of
$\lambda$, in terms of the single-box factor
\begin{equation}\label{eq:Nfactor}
\mathcal N(u)
:=
\frac{(1-[1,1]/u)(1-[0,0]/u)}{(1-[0,1]/u)(1-[1,0]/u)}
.
\end{equation}

\begin{lemma}\label{lem:A-corner-factor}
\[
\frac{A_\lambda(u)}{A_\lambda(u\Mt^{-1})}
=
\prod_{b\in\lambda}\mathcal N(u/[b])^{-1}.
\]
Equivalently,
\[
\frac{A_\lambda(u)}{A_\lambda(u\Mt^{-1})}
=
\frac{\,\,\,\,\,\,\prod\limits_{a\in\A_\lambda}\,\,\,\,\,\,(1-[a]/u)}
{\prod\limits_{b\in\R_\lambda^+\cup \{(0,0)\}}(1-[b]/u)}.
\]
\end{lemma}
\begin{proof}
For the boxes in row $i-1$,
\begin{align*}
\prod_{j=0}^{\lambda_i-1}\mathcal N(u/[i-1,j])^{-1}
&=
\frac{(1-\Mt^{i-1}\Mq^{-\lambda_i}/u)(1-\Mt^i/u)}
     {(1-\Mt^{i-1}/u)(1-\Mt^i\Mq^{-\lambda_i}/u)}\\
&=
\frac{(\Mq^{-\lambda_i}-u\Mt^{1-i})(1-u\Mt^{-i})}
     {(1-u\Mt^{1-i})(\Mq^{-\lambda_i}-u\Mt^{-i})}.
\end{align*}
This is the $i$th factor of $A_\lambda(u)/A_\lambda(u\Mt^{-1})$, and the
product over $i$ gives the first identity.  In the product over boxes the
factors shared by adjacent boxes cancel, which leaves the second form.
\end{proof}

\section{The full Lax operator}\label{sec:reformulation}

We now reformulate the results of the previous section in the framework of
Mickler--Moll \cite{MicklerMoll:2023} and Mickler \cite{Mickler:2023}.

\subsection{Definition and block form}

We define an inner product on the full space $\Lambda[z^{-1}]$ by
\begin{equation}\label{def:innerproductinf}
\langle f, g \rangle_\infty = \mathrm{CT}_z\langle f(\Mq^{-1} z),g(z^{-1})\rangle_{\Mq,\Mt},
\end{equation}
where $\mathrm{CT}_z$ takes the constant term in $z$. Equivalently,
$\langle z^{-m} p_\mu, z^{-n} p_\nu \rangle_\infty = \Mq^m \delta_{mn}\langle p_\mu,p_\nu\rangle_{\Mq,\Mt}$. Consider the following operator on $\Lambda[z^{-1}]$,
\[ (\mathbf{Q}f)(z) = \pi_{\le 0}Q(z)f(z)=\pi_{\le 0}\left(\exp\left(\sum_{n\geq1}\frac{1-\Mt^n}{n}p_nz^n\right)f(z)\right),\]
with formal adjoint with respect to \eqref{def:innerproductinf} given by
\[ \mathbf{Q}^\dagger= Q^*(\Mq^{-1}z^{-1}) =\exp\left(-\sum_{n\geq1}\frac{1-\Mq^{-n}}{n}p_n^\perp z^{-n}\right)\in \operatorname{End}(\Lambda[z^{-1}]).\]

\begin{definition}[Full Macdonald Lax operator; compare {\cite[Definition~1.1]{MicklerMoll:2023}}]\label{def:Laxdef}
The Macdonald Lax operator is
\begin{equation}\label{eq:Laxdef}
\LL:=\,\mathbf{Q}^\dagger\,\D_{\Mq^{-1}\Mt}\,\mathbf{Q}\quad \in \operatorname{End}(\Lambda[z^{-1}]).
\end{equation}
\end{definition}
The operator $\LL$ is formally self-adjoint with respect to
\eqref{def:innerproductinf} and preserves the total degree, with
$\deg z^{-1}=1$.
This operator does not appear in \cite{Nazarov:2015ab}, but its relation to their
scalar resolvent \eqref{def:resolvent} is explained in
Section~\ref{ssec:transfer}.
We relate $\LL$ to the following \emph{modified Nazarov--Sklyanin Lax
operator} on $\La^\infty$.
\begin{definition}
\[
\LL^+:=\D_\Mq^{-1}\Lc^{NS} \in \operatorname{End}(\La^\infty).
\]
\end{definition}

\begin{lemma}[Shift property; compare {\cite[Corollary~2.5]{Mickler:2023}}]\label{lem:shift}
As operators on $\La[z^{-1}]$,
\[
z\,\LL^+\,z^{-1}=\Mq^{-1}\Mt\,\LL .
\]
\end{lemma}
\begin{proof}
We start from \(\LL^+=\D_\Mq^{-1}\Lc^{NS}=\D_\Mq^{-1}Q^*(z^{-1})\D_\Mt\pi_-Q(z)\). The series
\(Q(z)\) and \(Q^*(z^{-1})\) commute with multiplication by \(z\), so conjugation
by \(z(\cdot)z^{-1}\) only acts on the diagonal operators and the projection:
\[
z\,\D_\Mq^{-1}\,z^{-1}=\Mq^{-1}\D_\Mq^{-1},\qquad
z\,\D_\Mt\,z^{-1}=\Mt\,\D_\Mt,\qquad
z\,\pi_-\,z^{-1}=\pi_{\le 0}.
\]
Therefore
\[
z\,\LL^+\,z^{-1}
=
\Mq^{-1}\D_\Mq^{-1}\,Q^*(z^{-1})\,\Mt\,\D_\Mt\,\pi_{\le 0}\,Q(z)
=
\Mq^{-1}\Mt\,\D_\Mq^{-1}\,Q^*(z^{-1})\,\D_\Mt\,\pi_{\le 0}\,Q(z)
=
\Mq^{-1}\Mt\,\LL,
\]
where the last step uses $\D_{\Mq^{-1}\Mt}=\D_\Mq^{-1}\D_\Mt$ and
$
Q^*(\Mq^{-1}z^{-1})\D_\Mq^{-1}
=
\D_\Mq^{-1}Q^*(z^{-1}).
$
\end{proof}

\begin{lemma}[Block form of $\LL$]\label{lem:blockform}
With respect to \(\La[z^{-1}]=\La\oplus\La^\infty\), we have
\[
\LL
=
\begin{pmatrix}
1 & \Q\\
\D_\Mq^{-1}\Q^* & \D_\Mq^{-1}(\Q^*\Q+\Lc^{NS})
\end{pmatrix}
\]
\end{lemma}
\begin{proof}
By the previous lemma,
\[
\LL=(\Mq\Mt^{-1})z\D_\Mq^{-1}\Lc^{NS}z^{-1}.
\]
For \(f\in\La\) and \(\chi\in\La^\infty\),
\[
\pi_-Q(z)z^{-1}(f+\chi)
=
z^{-1}\left(f+\Q\chi+\pi_-Q(z)\chi\right).
\]
Using \(\D_\Mt z^{-1}=\Mt z^{-1}\D_\Mt\) and
\[
\Mq z\D_\Mq^{-1}z^{-1}=\D_\Mq^{-1},
\]
we get
\[
\LL(f+\chi)
=
\D_\Mq^{-1}Q^*(z^{-1})\D_\Mt
\left(f+\Q\chi+\pi_-Q(z)\chi\right).
\]
Since \(f+\Q\chi\in\La\), this equals
\[
\begin{pmatrix}
f+\Q\chi\\
\D_\Mq^{-1}\Q^*(f+\Q\chi)
+\D_\Mq^{-1}\Lc^{NS}\chi
\end{pmatrix}.
\]
\end{proof}

Although $\LL$ is not multiplicative, it is multiplicative when one factor
lies in $\La$.

\begin{lemma}\label{lem:pi0homo}
For $\eta\in\La=\pi_0\La[z^{-1}]$ and $b\in\La[z^{-1}]$,
\[
\LL(\eta b)=\LL(\eta)\LL(b).
\]
\end{lemma}
\begin{proof}
Since $\eta$ has no $z$-dependence, $\mathbf Q(\eta b)=\eta\mathbf Qb$ and
$\D_{\Mq^{-1}\Mt}\eta=\eta$.  The operator $\mathbf Q^\dagger$ is an exponential
of commuting derivations and hence preserves multiplication.  Therefore
\[
\LL(\eta b)
=\mathbf Q^\dagger\bigl(\eta\D_{\Mq^{-1}\Mt}\mathbf Qb\bigr)
=(\mathbf Q^\dagger\eta)\LL(b)
=\LL(\eta)\LL(b).
\]
\end{proof}

\subsection{Transfer operators and Kerov measures}\label{ssec:transfer}

We now make the connection to the integrable hierarchy, via the transfer function, which was one of the motivations of Nazarov--Sklyanin.

\begin{definition}[compare {\cite[(2.2)]{MicklerMoll:2023}}]
The \emph{transfer operator} is
\[
\mathcal{T}(u) := \pi_0 (1-u^{-1}\LL)^{-1}\big|_{\La}
=\sum_{k=0}^{\infty} u^{-k} \mathcal{T}_k ,
\qquad \mathcal T_k\in\operatorname{End}(\La).
\]
\end{definition}

\begin{proposition}[compare {\cite[(2.3)]{MicklerMoll:2023}}]\label{cor:transfer-hierarchy}
For every partition $\lambda$,
\[
\mathcal{T}(u) J_\lambda = T_\lambda(u)J_\lambda,
\qquad
T_\lambda(u):=1+\frac{S_\lambda(u)}{u-1},
\qquad
S_\lambda(u) := \frac{A_\lambda(u\Mt^{-1})}{A_\lambda(u)} = \prod_{b\in\lambda}\mathcal N(u/[b]),
\]
and
\begin{equation}\label{eq:Tcorner}
T_\lambda(u)=1+u^{-1}\,
\frac{\prod_{b\in\R_\lambda^+}(1-[b]/u)}{\prod_{a\in\A_\lambda}(1-[a]/u)} .
\end{equation}
In particular the operators $\mathcal T_k$ commute and are diagonal on the
Macdonald polynomials.
\end{proposition}
\begin{proof}
Set
\[
F=u^{-1}(1-u^{-1}\LL)^{-1}J_\lambda=f+\chi,
\qquad
f=\pi_0F,\quad \chi=\pi_-F.
\]
Using the block form of \(\LL\), the equation \((u-\LL)F=J_\lambda\)
is equivalent to
$
(u-1)f-\Q\chi=J_\lambda
$
and
\[
(u\D_\Mq-\Lc^{NS})\chi
=
\Q^*(f+\Q\chi).
\]
The first equation gives \(f+\Q\chi=uf-J_\lambda\).  Substituting this into
the second equation and applying \(\Q(u\D_\Mq-\Lc^{NS})^{-1}\) gives
$
\Q\chi=I(u)(uf-J_\lambda).
$
Substitution back into the first equation gives
\[
\bigl((u-1)-uI(u)\bigr)f=(1-I(u))J_\lambda.
\]
By Theorem~\ref{thm:nsmain} and \eqref{eq:Aeigen}, \(I(u)\) is diagonal in the Macdonald basis, and
on \(J_\lambda\)
\[
I(u)=\frac{u-1}{u}\bigl(1-S_\lambda(u)^{-1}\bigr).
\]
Hence
\[
(u-1)-uI(u)=(u-1)S_\lambda(u)^{-1},
\qquad
1-I(u)=\frac{1+(u-1)S_\lambda(u)^{-1}}{u}.
\]
Therefore
\[
f
=
\frac{1+(u-1)S_\lambda(u)^{-1}}
{u(u-1)S_\lambda(u)^{-1}}
J_\lambda,
\qquad
\mathcal{T}(u)J_\lambda
=
u f
=
\left(
1+\frac{S_\lambda(u)}{u-1}
\right)J_\lambda
=
T_\lambda(u)J_\lambda .
\]
The corner form \eqref{eq:Tcorner} follows from
$S_\lambda(u)=(1-1/u)\prod_{b\in\R^+_\lambda}(1-[b]/u)/\prod_{a\in\A_\lambda}(1-[a]/u)$,
which is Lemma~\ref{lem:A-corner-factor}.
\end{proof}

\begin{definition}[Kerov {\cite[Definitions~2.3 and~2.8]{Kerov:2000}}]\label{def:kerov}
The \emph{transition measure}  of $\lambda$ is the family
$\{\tau_\lambda^s\}_{s\in\A_\lambda}$ determined by
\[
T_\lambda(u)
=
\sum_{s\in\A_\lambda}\frac{\tau_\lambda^s}{1-[s]/u} .
\]
Explicitly,
\[
\tau_\lambda^s
=\frac{1}{[s]}\,
\frac{\prod_{b\in\R^+_\lambda}(1-[b-s])}
{\prod_{a\in\A_\lambda,\,a\ne s}(1-[a-s])}.
\]
The \emph{co-transition measure} at an outer corner $b\in\R^+_\lambda$ is
\[
\tilde\tau_\lambda^{\,b}
:=\frac{\prod_{a\in\A_\lambda}(1-[a-b])}
{\prod_{b'\in\R^+_\lambda,\,b'\ne b}(1-[b'-b])}.
\]
\end{definition}
By \eqref{eq:Tcorner} and $|\A_\lambda|=|\R^+_\lambda|+1$,
\begin{equation}\label{eq:Tminus1}
T_\lambda(u)-1=\frac{\prod_{b\in\R^+_\lambda}(u-[b])}{\prod_{a\in\A_\lambda}(u-[a])},
\end{equation}
so both measures are residues:
\begin{equation}\label{eq:cotau-residue}
\tau_\lambda^s=\frac{1}{[s]}\operatorname*{Res}_{u=[s]}T_\lambda(u),
\qquad
\tilde\tau_\lambda^{\,b}=\frac{1}{[b]^2}\operatorname*{Res}_{u=[b]}\frac{1}{T_\lambda(u)-1}.
\end{equation}
The expansion defining $\tau_\lambda^s$ holds because $T_\lambda$ has simple
poles at the $[s]$, $T_\lambda(\infty)=1$ and $T_\lambda(0)=0$, the last since
$\prod_{b\in\R^+_\lambda}[b]=\prod_{a\in\A_\lambda}[a]$ in the corner
coordinates of Section~\ref{ssec:notation}.

\begin{remark}\label{rem:kerov}
We follow Kerov (see also
\cite{Kerov:1998}) and \cite{MicklerMoll:2023} in these names: the transition
measure lives on the addable corners and the co-transition measure on the outer
corners.  Applied to \eqref{eq:Tminus1},
his expansion \cite[(2.1)]{Kerov:2000} would give the weights $[s]\tau_\lambda^s$; both families have
total mass $1$ (Lemma~\ref{lem:kerovmoments}), and they agree in the Jack limit,
where every $[s]\to1$.  Likewise $\tilde\tau_\lambda^{\,b}=-\nu_b/[b]^2$,
where $1/(T_\lambda(u)-1)=u-c-\sum_{b\in\R^+_\lambda}\nu_b/(u-[b])$ is Kerov's
expansion \cite[(2.5)]{Kerov:2000}.  Kerov normalizes the $\nu_b$ to total
mass $1$; we do not.
\end{remark}

\begin{lemma}[Kerov moments]\label{lem:kerovmoments}
For every partition $\lambda$,
\[
\sum_{a\in\A_\lambda}\tau_\lambda^a=1,
\qquad
\sum_{a\in\A_\lambda}[a]\,\tau_\lambda^a=1,
\]
and, as identities of rational functions of $u$,
\[
\sum_{a\in\A_\lambda}\frac{\tau_\lambda^a[a]}{1-[a]/u}=u(T_\lambda(u)-1),
\qquad
\frac{1}{u(T_\lambda(u)-1)}=1-\frac1u
+\sum_{b\in\R^+_\lambda}\frac{[b]\,\tilde\tau_\lambda^{\,b}}{u-[b]}.
\]
For $s\in\A_\lambda$,
\begin{equation}\label{eq:pfvalues}
\sum_{b\in\R^+_\lambda}\frac{[b]\,\tilde\tau_\lambda^{\,b}}{[b]-[s]}=1-\frac1{[s]}.
\end{equation}
\end{lemma}
\begin{proof}
By \eqref{eq:Tcorner}, $T_\lambda(u)=1+u^{-1}+O(u^{-2})$ as $u\to\infty$, while
Definition~\ref{def:kerov} gives
$T_\lambda(u)=\sum_a\tau_\lambda^a\,u/(u-[a])
=\sum_a\tau_\lambda^a+u^{-1}\sum_a[a]\tau_\lambda^a+O(u^{-2})$.  Comparing
coefficients gives the first two identities.  The third follows from
$[a]/(1-[a]/u)=u/(1-[a]/u)-u$, Definition~\ref{def:kerov} and the first.
The left side of the fourth is $1$ at $u=\infty$, has residue $-1$ at $u=0$
because $T_\lambda(0)=0$, and has residue $[b]\tilde\tau_\lambda^{\,b}$ at
$u=[b]$ by \eqref{eq:cotau-residue}.  Finally, \eqref{eq:pfvalues} is the
fourth identity at $u=[s]$, where $T_\lambda$ has a pole.
\end{proof}

\subsection{Spectral theorem}

\begin{proposition}[Spectral theorem for $\LL$; compare {\cite[Theorem~1.4]{MicklerMoll:2023}}]\label{prop:spectral}
For a partition $\lambda\ne\varnothing$ and $s \in \A_\lambda$, let
$\Ps_\lambda^s\in\La[z^{-1}]$ be the value at $u=[s]$ of the $\La[z^{-1}]$-valued
rational function of $u$
\[ J_\lambda + \bigl(\D_\Mq-\Lc^{NS}/u\bigr)^{-1}\Q^*J_\lambda . \]
For the empty
partition, set $\Ps_\varnothing^{(0,0)}=1$.  Then $\Ps_\lambda^s$ is well
defined,
\[
\LL \Ps_\lambda^{s}
=
[s] \Ps_\lambda^{s},
\qquad
\pi_0\Ps_\lambda^s=J_\lambda,
\]
and the vectors $\Ps_\lambda^s$, over all partitions $\lambda$ and all
$s\in\A_\lambda$, form a basis of $\La[z^{-1}]$.
\end{proposition}
\begin{proof}
Work in a fixed total degree.  The empty-partition case is immediate.  Let
$\lambda\ne\varnothing$.  By Theorem~\ref{thm:nsmain} and \eqref{eq:Aeigen},
\[
uI(u)J_\lambda=(u-1)\left(1-\frac{A_\lambda(u)}{A_\lambda(u\Mt^{-1})}\right)J_\lambda .
\]
By Lemma~\ref{lem:A-corner-factor}, the poles of
$A_\lambda(u)/A_\lambda(u\Mt^{-1})$ lie at the contents $[b]$ with
$b\in\R^+_\lambda\cup\{(0,0)\}$, and the pole at $[0,0]=1$ is cancelled by the
factor $u-1$.  Since $s\notin\R^+_\lambda\cup\{(0,0)\}$ and distinct boxes have
distinct contents, $uI(u)J_\lambda$ is regular at $u=[s]$.  Moreover the ratio
vanishes at $u=[s]$, so the value there is $([s]-1)J_\lambda$.

By the block form, the
adjoint of $\Q$ for $\ip{\cdot}{\cdot}_\infty$ is $\D_\Mq^{-1}\Q^*$, and $\LL^+$
is self-adjoint, being the lower-right block of $\LL$ minus
$\D_\Mq^{-1}\Q^*\Q$.  Temporarily choose real $0<\Mq,\Mt<1$ away from poles and
content coincidences, so that the scalar product is positive definite.  Put
$v_\lambda=\D_\Mq^{-1}\Q^* J_\lambda$.  Then
\[
\langle J_\lambda,I(u)J_\lambda\rangle_{\Mq,\Mt}
=\langle v_\lambda,(u-\LL^+)^{-1}v_\lambda\rangle_\infty.
\]
In the spectral decomposition of $\LL^+$, the residue on the right at each
eigenvalue is the squared norm of the corresponding component of
$v_\lambda$.  Regularity at $u=[s]$ thus implies that
$(u-\LL^+)^{-1}v_\lambda$ is regular there as well.  This holds over
$\QQ(\Mq,\Mt)$: the entries of $(u-\LL^+)^{-1}v_\lambda$ are rational functions of
$\Mq,\Mt,u$, and a pole of an entry along $u=[s]$ would persist after specializing
$(\Mq,\Mt)$ to a Zariski-dense set of points, hence to points of the open set
chosen above.  Consequently the rational vector
\[
\psi(u):=u(u-\LL^+)^{-1}v_\lambda=(\D_\Mq-\Lc^{NS}/u)^{-1}\Q^*J_\lambda
\]
is regular at $u=[s]$.

Now write $\Ps_\lambda^s=J_\lambda+\psi$ with $\psi=\psi([s])\in\La^\infty$, so
that $\pi_0\Ps_\lambda^s=J_\lambda$.  Since $\Q\psi(u)=uI(u)J_\lambda$ for
generic $u$, the first paragraph gives
\[
\Q\psi=([s]-1)J_\lambda .
\]
Moreover, $(\D_\Mq-\Lc^{NS}/u)\psi(u)=\Q^*J_\lambda$ identically in $u$, so
$\D_\Mq^{-1}\Lc^{NS}\psi(u)=u\psi(u)-u\D_\Mq^{-1}\Q^*J_\lambda$, and evaluating at
$u=[s]$ gives
\[
\D_\Mq^{-1}\Lc^{NS}\psi=[s]\psi-[s]\D_\Mq^{-1}\Q^*J_\lambda .
\]
By the block form, the $\La$-component of $\LL\Ps_\lambda^s$ is
$J_\lambda+\Q\psi=[s]J_\lambda$, and its $\La^\infty$-component is
\[
\D_\Mq^{-1}\Q^*J_\lambda+\D_\Mq^{-1}\Q^*\Q\psi+\D_\Mq^{-1}\Lc^{NS}\psi
=\D_\Mq^{-1}\Q^*J_\lambda+([s]-1)\D_\Mq^{-1}\Q^*J_\lambda
+[s]\psi-[s]\D_\Mq^{-1}\Q^*J_\lambda
=[s]\psi .
\]
Hence $\LL\Ps_\lambda^s=[s]\Ps_\lambda^s$.

It remains to prove that these eigenvectors form a basis.  Fix a total degree
$N$ and suppose that
\[
\sum_{\lambda\vdash N}\sum_{s\in\A_\lambda}
 c_{\lambda,s}\Ps_\lambda^s=0.
\]
For each corner content $a$ occurring in degree $N$, apply
$\prod_{b\ne a}(\LL-b)$, the product over the other corner contents $b$
occurring in degree $N$, and then project with $\pi_0$.  After dividing by the resulting nonzero common
scalar, and using $\pi_0\Ps_\lambda^s=J_\lambda$, this gives
\[
\sum_{\substack{\lambda\vdash N,\ s\in\A_\lambda\\ [s]=a}}
 c_{\lambda,s}J_\lambda=0.
\]
The addable corners of a fixed partition have distinct contents, so at most one
term occurs for each $J_\lambda$.  Their linear independence implies
$c_{\lambda,s}=0$ for every pair $(\lambda,s)$.

It remains only to count the vectors.  If $p(N)$ is the number of partitions of
$N$, then
\[
\sum_{N\ge0}\left(\sum_{\lambda\vdash N}|\A_\lambda|\right)x^N
=\frac{1}{1-x}\prod_{m\ge1}\frac{1}{1-x^m}.
\]
Indeed, $|\A_\lambda|$ is one plus the number of distinct positive parts of
$\lambda$.  Hence
\[
\sum_{\lambda\vdash N}|\A_\lambda|
=\sum_{k=0}^{N}p(k)
=\dim\La[z^{-1}]_N.
\]
\end{proof}

\begin{proposition}[compare {\cite[(1.17)]{MicklerMoll:2023}}]\label{prop:jackpsiexp}
For every partition $\lambda$,
\begin{equation}\label{eq:jackpsiexp}
J_\lambda=\sum_{s\in\A_\lambda}\tau_\lambda^s\Ps_\lambda^s .
\end{equation}
\end{proposition}
\begin{proof}
Expand $J_\lambda=\sum_{\mu\vdash|\lambda|}\sum_{r\in\A_\mu}c_{\mu,r}\Ps_\mu^r$ in
the basis of Proposition~\ref{prop:spectral}.  Applying $\pi_0(1-\LL/u)^{-1}$
and using Proposition~\ref{cor:transfer-hierarchy},
\[
T_\lambda(u)J_\lambda=\sum_{\mu\vdash|\lambda|}\Bigl(\sum_{r\in\A_\mu}\frac{c_{\mu,r}}{1-[r]/u}\Bigr)J_\mu .
\]
The addable corners of $\mu$ have distinct contents, so the functions
$(1-[r]/u)^{-1}$, $r\in\A_\mu$, are linearly independent.  Comparing
coefficients of $J_\mu$ gives $c_{\mu,r}=0$ for $\mu\ne\lambda$, and for
$\mu=\lambda$ Definition~\ref{def:kerov} gives $c_{\lambda,s}=\tau_\lambda^s$.
\end{proof}

\begin{proposition}\label{prop:orthogonal}
The basis $\{\Ps_\lambda^s\}$ of Proposition~\ref{prop:spectral} is orthogonal
for $\ip{\cdot}{\cdot}_\infty$.  More precisely, for a partition $\lambda$ put
\[
Z_\lambda:=\{f(\LL)J_\lambda:f\in\QQ(\Mq,\Mt)[w]\}.
\]
Then $Z_\lambda=\operatorname{span}\{\Ps_\lambda^s:s\in\A_\lambda\}$, the
spaces $Z_\lambda$ are mutually orthogonal, and within each $Z_\lambda$ the
vectors $\Ps_\lambda^s$ are orthogonal.
\end{proposition}
\begin{proof}
The argument follows the cyclic-space proof of
\cite[Section~3]{MicklerMoll:2023}.  Since the eigenvalues $[s]$,
$s\in\A_\lambda$, are distinct, Proposition~\ref{prop:jackpsiexp} and Lagrange
interpolation give
\[
\Ps_\lambda^s
=\frac{1}{\tau_\lambda^s}
 \prod_{\substack{r\in\A_\lambda\\r\ne s}}
 \frac{\LL-[r]}{[s]-[r]}J_\lambda,
\]
so $Z_\lambda=\operatorname{span}\{\Ps_\lambda^s:s\in\A_\lambda\}$.  Here
$\tau_\lambda^s\ne0$, since no factor in its formula in
Definition~\ref{def:kerov} vanishes: distinct boxes have distinct contents.  By
Proposition~\ref{cor:transfer-hierarchy}, $\pi_0f(\LL)J_\lambda\in\QQ(\Mq,\Mt)J_\lambda$
for every polynomial $f$.  If $\lambda\ne\mu$, $F=f(\LL)J_\lambda\in Z_\lambda$
and $G=g(\LL)J_\mu\in Z_\mu$, then self-adjointness of $\LL$ gives
\begin{align*}
\langle F,G\rangle_\infty
&=\langle J_\lambda,f(\LL)g(\LL)J_\mu\rangle_\infty\\
&=\langle J_\lambda,\pi_0f(\LL)g(\LL)J_\mu\rangle_{\Mq,\Mt}=0,
\end{align*}
because the $J_\lambda$ are orthogonal.  Thus $Z_\lambda\perp Z_\mu$ for
$\lambda\ne\mu$.  Within a fixed $Z_\lambda$, the vectors $\Ps_\lambda^s$ have
distinct eigenvalues and are therefore orthogonal by self-adjointness of
$\LL$.
\end{proof}

\begin{corollary}[Spectral theorem for $\LL^+$; compare {\cite[Theorem~2.11]{Mickler:2023}}]\label{cor:NSspectral}
In each total degree $N\ge1$, the vectors $z^{-1}\Ps_{\lambda-r}^{\,r}$,
indexed by $\lambda\vdash N$ and $r\in\R_\lambda$, form an orthogonal basis of
$\La^\infty_N$ for $\ip{\cdot}{\cdot}_\infty$, and
\[
\LL^+z^{-1}\Ps_{\lambda-r}^{\,r}=[r^+]\,z^{-1}\Ps_{\lambda-r}^{\,r}.
\]
Equivalently, these vectors solve the
generalized eigenvalue problem $\Lc^{NS}\varphi=[r^+]\,\D_\Mq\varphi$.
\end{corollary}
\begin{proof}
By Lemma~\ref{lem:shift}, $\LL^+=(\Mt/\Mq)\,z^{-1}\LL z$ on $\La^\infty$, and
$(\Mt/\Mq)[r]=[r^+]$.  The pairs $(\lambda-r,r)$ run over all pairs
$(\gamma,c)$ with $\gamma\vdash N-1$ and $c\in\A_\gamma$.  Multiplication by
$z^{-1}$ is an isomorphism $\La[z^{-1}]_{N-1}\to\La^\infty_N$ with
$\ip{z^{-1}F}{z^{-1}G}_\infty=\Mq\ip{F}{G}_\infty$, so the claim follows from
Propositions~\ref{prop:spectral} and~\ref{prop:orthogonal}.
\end{proof}

\begin{lemma}\label{lem:transmeas}
The transition measure admits the hook product form
\begin{equation}\label{eq:tauhooks}
\tau_\lambda^s
=
\Mq^{a'(s)}
\prod_{\substack{b\in\lambda\\ b\in \mathrm{Row}_s(\lambda)}}
\frac{h_\lambda(b)}{h_{\lambda+s}(b)}
\prod_{\substack{b\in\lambda\\ b\in \mathrm{Col}_s(\lambda)}}
\frac{h'_\lambda(b)}{h'_{\lambda+s}(b)},
\end{equation}
where \(\mathrm{Row}_s(\lambda)\) is the set of boxes of \(\lambda\) in the same row as
\(s\), \(\mathrm{Col}_s(\lambda)\) the set of boxes in the same column as \(s\), and
$h,h'$ are the hook factors of Section~\ref{ssec:notation}.
\end{lemma}
\begin{proof}
In the corner coordinates of Section~\ref{ssec:notation}, write
\[
a_k=(I_k,L_{k+1})\quad(0\le k\le m),
\qquad
r_k=(I_k,L_k)\quad(1\le k\le m)
\]
for the addable and outer corners.  Fix
\[
s=a_p=(I_p,L_{p+1})=(i,j).
\]
In these coordinates, the formula for $\tau_\lambda^s$ in
Definition~\ref{def:kerov} reads
\begin{equation}
\tau_\lambda^s
=\frac{1}{[s]}
 \frac{\displaystyle\prod_{k=1}^{m}(1-[r_k]/[s])}
      {\displaystyle\prod_{\substack{0\le k\le m\\k\ne p}}
       (1-[a_k]/[s])}.
\label{eq:tau-corner-product}
\end{equation}

We now compute the two hook products in \eqref{eq:tauhooks}.  First consider the
boxes $(i,c)$ in the row of $s$.  For
\[
L_{k+1}\le c<L_k,
\qquad p+1\le k\le m,
\]
one has $\ell_\lambda(i,c)=I_k-I_p-1$, while adding $s$ increases the arm by
one.  Hence the factors in each such block telescope:
\begin{align*}
\prod_{c=L_{k+1}}^{L_k-1}
\frac{h_\lambda(i,c)}{h_{\lambda+s}(i,c)}
&=
\frac{1-\Mq^{j-L_k}\Mt^{I_k-I_p}}
     {1-\Mq^{j-L_{k+1}}\Mt^{I_k-I_p}}\\
&=
\frac{1-[r_k]/[s]}{1-[a_k]/[s]}.
\end{align*}
Therefore
\begin{equation}
\prod_{b\in \mathrm{Row}_s(\lambda)}
\frac{h_\lambda(b)}{h_{\lambda+s}(b)}
=
\prod_{k=p+1}^{m}
\frac{1-[r_k]/[s]}{1-[a_k]/[s]}.
\label{eq:tau-row-product}
\end{equation}

Next consider the boxes $(c,j)$ in the column of $s$.  For
\[
I_{k-1}\le c<I_k,
\qquad 1\le k\le p,
\]
one has $a_\lambda(c,j)=L_k-j-1$, while adding $s$ increases the leg by one.
Thus
\begin{align*}
\prod_{c=I_{k-1}}^{I_k-1}
\frac{h'_\lambda(c,j)}{h'_{\lambda+s}(c,j)}
&=
\frac{1-\Mq^{L_k-j}\Mt^{I_p-I_k}}
     {1-\Mq^{L_k-j}\Mt^{I_p-I_{k-1}}}\\
&=
\frac{1-[s]/[r_k]}{1-[s]/[a_{k-1}]}\\
&=
\frac{[a_{k-1}]}{[r_k]}
\frac{1-[r_k]/[s]}{1-[a_{k-1}]/[s]}.
\end{align*}
Since
\[
\prod_{k=1}^{p}\frac{[a_{k-1}]}{[r_k]}
=
\prod_{k=1}^{p}\Mt^{I_{k-1}-I_k}
=\Mt^{-I_p}=\Mt^{-i},
\]
we obtain
\begin{equation}
\prod_{b\in \mathrm{Col}_s(\lambda)}
\frac{h'_\lambda(b)}{h'_{\lambda+s}(b)}
=
\Mt^{-i}
\prod_{k=1}^{p}
\frac{1-[r_k]/[s]}{1-[a_{k-1}]/[s]}.
\label{eq:tau-column-product}
\end{equation}

Multiplying \eqref{eq:tau-row-product} and
\eqref{eq:tau-column-product} accounts for every shifted removable corner and
every addable corner except $a_p=s$.  Since $[s]=\Mt^i\Mq^{-j}$,
\eqref{eq:tau-corner-product} becomes
\[
\tau_\lambda^s
=\Mq^j
\prod_{b\in \mathrm{Row}_s(\lambda)}
\frac{h_\lambda(b)}{h_{\lambda+s}(b)}
\prod_{b\in \mathrm{Col}_s(\lambda)}
\frac{h'_\lambda(b)}{h'_{\lambda+s}(b)}.
\]
Finally $j=a'(s)$, proving \eqref{eq:tauhooks}.
\end{proof}

\begin{lemma}\label{lem:norm}
Let \(s\in\A_\lambda\).  Then
\begin{align*}
\|\Ps_\lambda^s\|_\infty^2
&=\frac{\|J_\lambda\|_{\Mq,\Mt}^2}{\tau_\lambda^s}\\
&=\Mq^{-a'(s)}
\prod_{b\in\lambda\setminus(\mathrm{Row}_s(\lambda)\cup \mathrm{Col}_s(\lambda))}
 h_\lambda(b)h'_\lambda(b)
\prod_{b\in \mathrm{Row}_s(\lambda)}h_{\lambda+s}(b)h'_\lambda(b)
\prod_{b\in \mathrm{Col}_s(\lambda)}h_\lambda(b)h'_{\lambda+s}(b).
\end{align*}

\end{lemma}
\begin{proof}
Since \(\pi_0\Ps_\lambda^s=J_\lambda\), grading gives
\[
\langle \Ps_\lambda^s,J_\lambda\rangle_\infty
=
\langle J_\lambda,J_\lambda\rangle_{\Mq,\Mt}.
\]
By \eqref{eq:jackpsiexp} and Proposition~\ref{prop:orthogonal},
$\langle \Ps_\lambda^s,J_\lambda\rangle_\infty=\tau_\lambda^s\|\Ps_\lambda^s\|_\infty^2$,
which gives the first equality.
Substituting the hook norm
$\|J_\lambda\|_{\Mq,\Mt}^2=\prod_{b\in\lambda}h_\lambda(b)h'_\lambda(b)$
and the transition-measure formula \eqref{eq:tauhooks}, then cancelling the row
and column factors, gives the second equality.
\end{proof}

\begin{example}[Low-degree eigenvectors]\label{ex:lowdegree}
Write
\[
\Ps_\lambda^s=J_\lambda+\psi_\lambda^s,
\qquad
\psi_\lambda^s\in z^{-1}\La[z^{-1}].
\]

For \(|\lambda|=1\),
\[
\begin{aligned}
\psi_{(1)}^{(0,1)}
&=-\Mq^{-1}(\Mq-1)z^{-1},\\
\psi_{(1)}^{(1,0)}
&=(\Mt-1)z^{-1}.
\end{aligned}
\]

For \(|\lambda|=2\),
\[
\begin{aligned}
\psi_{(2)}^{(0,2)}
&=-\Mq^{-1}(\Mq-1)(\Mq+1)z^{-1}J_{(1)}
+\Mq^{-2}(\Mq+1)(\Mq-1)^2z^{-2},\\
\psi_{(2)}^{(1,0)}
&=\Mq(\Mt-1)z^{-1}J_{(1)}
-(\Mt-1)(\Mq-1)z^{-2},\\
\psi_{(1,1)}^{(0,1)}
&=-\Mq^{-1}(\Mq-1)z^{-1}J_{(1)}
-\Mq^{-1}(\Mt-1)(\Mq-1)z^{-2},\\
\psi_{(1,1)}^{(2,0)}
&=(\Mt-1)(\Mt+1)z^{-1}J_{(1)}
+(\Mt+1)(\Mt-1)^2z^{-2}.
\end{aligned}
\]

For \(|\lambda|=3\),
\[
\begin{aligned}
\psi_{(3)}^{(0,3)}
&=-\Mq^{-1}(\Mq-1)(\Mq^2+\Mq+1)z^{-1}J_{(2)}\\
&\quad
+\Mq^{-2}(\Mq+1)(\Mq-1)^2(\Mq^2+\Mq+1)z^{-2}J_{(1)}\\
&\quad
-\Mq^{-3}(\Mq+1)(\Mq-1)^3(\Mq^2+\Mq+1)z^{-3},\\
\psi_{(3)}^{(1,0)}
&=\Mq^2(\Mt-1)z^{-1}J_{(2)}
-\Mq(\Mq+1)(\Mq-1)(\Mt-1)z^{-2}J_{(1)}\\
&\quad
+(\Mt-1)(\Mq+1)(\Mq-1)^2z^{-3}.
\end{aligned}
\]
\[
\begin{aligned}
\psi_{(2,1)}^{(0,2)}
&=-\Mq^{-1}(\Mq-1)(\Mq\Mt-1)^{-1}
\left((\Mt-1)J_{(2)}+(\Mq^2\Mt-1)J_{(1,1)}\right)z^{-1}\\
&\quad
+\Mq^{-2}(\Mq-1)(\Mq^2-\Mq\Mt+\Mq-1)z^{-2}J_{(1)}\\
&\quad
+\Mq^{-2}(\Mt-1)(\Mq+1)(\Mq-1)^2z^{-3},\\
\psi_{(2,1)}^{(1,1)}
&=-\Mq^{-1}(\Mq-1)(\Mq\Mt-1)^{-1}(\Mq\Mt^2-1)z^{-1}J_{(2)}\\
&\quad
+(\Mt-1)(\Mq\Mt-1)^{-1}(\Mq^2\Mt-1)z^{-1}J_{(1,1)}\\
&\quad
-\Mq^{-1}(\Mt-1)(\Mq-1)(\Mq\Mt+\Mq+1)z^{-2}J_{(1)}\\
&\quad
+\Mq^{-1}(\Mq-\Mt)(\Mt-1)(\Mq-1)z^{-3},\\
\psi_{(2,1)}^{(2,0)}
&=(\Mt-1)(\Mq\Mt-1)^{-1}(\Mq\Mt^2-1)z^{-1}J_{(2)}\\
&\quad
+\Mq\Mt(\Mt-1)(\Mq-1)(\Mq\Mt-1)^{-1}z^{-1}J_{(1,1)}\\
&\quad
+(\Mt-1)(\Mq\Mt^2-\Mq\Mt-\Mq+\Mt)z^{-2}J_{(1)}\\
&\quad
-(\Mt+1)(\Mq-1)(\Mt-1)^2z^{-3}.
\end{aligned}
\]
\[
\begin{aligned}
\psi_{(1,1,1)}^{(0,1)}
&=-\Mq^{-1}(\Mq-1)z^{-1}J_{(1,1)}
-\Mq^{-1}(\Mt-1)(\Mt+1)(\Mq-1)z^{-2}J_{(1)}\\
&\quad
-\Mq^{-1}(\Mt+1)(\Mq-1)(\Mt-1)^2z^{-3},\\
\psi_{(1,1,1)}^{(3,0)}
&=(\Mt-1)(\Mt^2+\Mt+1)z^{-1}J_{(1,1)}\\
&\quad
+(\Mt+1)(\Mt-1)^2(\Mt^2+\Mt+1)z^{-2}J_{(1)}\\
&\quad
+(\Mt+1)(\Mt-1)^3(\Mt^2+\Mt+1)z^{-3}.
\end{aligned}
\]

\end{example}

\section{Specialization and creation}\label{sec:ladder}

This section computes the principal specialization of the eigenvectors and
the action on the eigenbasis of multiplication by $z^{-1}$ and of the down map
$z\pi_-$.

\subsection{Specialization}\label{sec:specialization}

We use the principal specialization of Macdonald
\cite[Chapter~VI, (6.16)]{Macdonald:1995},
\[
\epsilon_v:\La\to\QQ(\Mq,\Mt,v),
\qquad
\epsilon_v(p_r)=\frac{1-v^r}{1-\Mt^r}.
\]

\begin{theorem}[Principal specialization; {\cite[Chapter~VI, (6.17)]{Macdonald:1995}}]\label{thm:principal-spec}
For a partition \(\lambda\),
\[
\epsilon_v(P_\lambda)
=
\Mt^{n(\lambda)}
\prod_{b\in\lambda}
\frac{1-v\Mq^{a'(b)}\Mt^{-\ell'(b)}}{1-\Mq^{a_\lambda(b)}\Mt^{\ell_\lambda(b)+1}}.
\]
Thus, for the integral form,
\[
\epsilon_v(J_\lambda)
=
\Mt^{n(\lambda)}
\prod_{b\in\lambda}
\left(1-v/[b]\right).
\]
\end{theorem}
Extend \(\epsilon_v\) to a ring homomorphism on \(\La[z^{-1}]\) by
\(\epsilon_v(z^{-1})=1\).

\begin{proposition}[compare {\cite[Proposition~4.3]{MicklerMoll:2023}}]\label{prop:epsPsi}
For every partition $\lambda$ and $s\in\A_\lambda$,
\[
\epsilon_v(\Ps_\lambda^s)
=
\frac{[s]-v}{1-v}\epsilon_v(J_\lambda).
\]
\end{proposition}
\begin{proof}
Write $\Ps_\lambda^s=J_\lambda+\psi_\lambda^s$ with
$\psi_\lambda^s=\pi_-\Ps_\lambda^s$.  By the block form (Lemma~\ref{lem:blockform}) and Proposition~\ref{prop:spectral},
\[
\Q\psi_\lambda^s=\pi_0(\LL-1)\Ps_\lambda^s=([s]-1)J_\lambda .
\]
Since $\epsilon_v(Q_m)=1-v$ for $m\ge1$,
$\epsilon_v(\Q\psi_\lambda^s)=(1-v)\epsilon_v(\psi_\lambda^s)$.  Hence
$\epsilon_v(\Ps_\lambda^s)=\epsilon_v(J_\lambda)+\epsilon_v(\psi_\lambda^s)
=\frac{[s]-v}{1-v}\,\epsilon_v(J_\lambda)$.
\end{proof}

\begin{corollary}\label{cor:eps-resolvent}
\[
\epsilon_v\left((u-\LL)^{-1}J_\lambda\right)
=
\frac{(1-v/u)T_\lambda(u)-1}{1-v}\epsilon_v(J_\lambda).
\]
\end{corollary}
\begin{proof}
Using \eqref{eq:jackpsiexp} and Proposition~\ref{prop:epsPsi},
\[
\epsilon_v\left((u-\LL)^{-1}J_\lambda\right)
=
\sum_{s\in\A_\lambda}
\frac{\tau_\lambda^s}{u-[s]}
\frac{[s]-v}{1-v}\epsilon_v(J_\lambda).
\]
By Lemma~\ref{lem:kerovmoments},
\[
\sum_{s\in\A_\lambda}\frac{\tau_\lambda^s[s]}{u-[s]}
=T_\lambda(u)-1,
\qquad
\sum_{s\in\A_\lambda}\frac{\tau_\lambda^s v}{u-[s]}
=\frac{v}{u}T_\lambda(u).
\]
Substituting these two sums proves the formula.
\end{proof}

The eigenvectors are normalized by $\pi_0\Ps_\lambda^s=J_\lambda$.  Let
\[
\pi_\infty:\La[z^{-1}]\to\QQ(\Mq,\Mt)[z^{-1}]
\]
set every $p_k$ to zero and leave $z^{-1}$ fixed.  Its restriction to $\La$ is
$\epsilon_1$.  For \(\lambda\neq\varnothing\), set
\[
\varpi_\lambda
:= \lim_{v\to1} \frac{\epsilon_v(J_\lambda)}{1-v}=
\Mt^{n(\lambda)}
\prod_{\substack{b\in\lambda\\ b\neq(0,0)}}
\left(1-1/[b]\right) .
\]
For $\lambda\ne\varnothing$ and $s\in\A_\lambda$, the definition gives
\begin{equation}\label{eq:varpi-ratio}
\varpi_{\lambda+s}=\Mq^{a'(s)}([s]-1)\,\varpi_\lambda .
\end{equation}
We also write
\[
\widehat J_\lambda:=\frac{J_\lambda}{\varpi_\lambda}.
\]

\begin{proposition}\label{prop:piinftypsi}

For \(\lambda\neq\varnothing\) and \(s\in\A_\lambda\),
\[
\pi_\infty\Ps_\lambda^s
=
([s]-1)\varpi_\lambda z^{-|\lambda|}.
\]
\end{proposition}
\begin{proof}
Since $\Ps_\lambda^s$ is homogeneous of total degree $|\lambda|$,
$\pi_\infty\Ps_\lambda^s=\epsilon_1(\Ps_\lambda^s)z^{-|\lambda|}$.  By
Proposition~\ref{prop:epsPsi} and the definition of $\varpi_\lambda$,
\[
\epsilon_1(\Ps_\lambda^s)
=\lim_{v\to1}\frac{[s]-v}{1-v}\epsilon_v(J_\lambda)
=([s]-1)\varpi_\lambda .
\]
\end{proof}

\begin{corollary}\label{cor:pi_infty_J}
For $\lambda\ne\varnothing$,
\[
\pi_\infty (1-\LL/u)^{-1}J_\lambda
=
\left(
\frac{A_\lambda(u\Mt^{-1})}{A_\lambda(u)}-1
\right)\varpi_\lambda z^{-|\lambda|}.
\]
\end{corollary}
\begin{proof}
Using \eqref{eq:jackpsiexp} and Proposition~\ref{prop:piinftypsi},
\[
\pi_\infty (u-\LL)^{-1}J_\lambda
=
\varpi_\lambda
\sum_{s\in\A_\lambda}
\frac{\tau_\lambda^s([s]-1)}{u-[s]}z^{-|\lambda|}.
\]
Now
\[
\sum_{s\in\A_\lambda}
\frac{\tau_\lambda^s([s]-1)}{u-[s]}
=
\sum_{s\in\A_\lambda}
\frac{\tau_\lambda^s[s]}{u-[s]}
-
\sum_{s\in\A_\lambda}
\frac{\tau_\lambda^s}{u-[s]}.
\]
By Lemma~\ref{lem:kerovmoments} and Definition~\ref{def:kerov}, the first sum
is \(T_\lambda(u)-1\) and the second is \(T_\lambda(u)/u\).  Hence the bracket is
\[
\frac{u-1}{u}T_\lambda(u)-1.
\]
Multiplying by $u$ and using $T_\lambda(u)=1+S_\lambda(u)/(u-1)$
(Proposition~\ref{cor:transfer-hierarchy}) gives
\[
(u-1)T_\lambda(u)-u=S_\lambda(u)-1=\frac{A_\lambda(u\Mt^{-1})}{A_\lambda(u)}-1 .
\]
\end{proof}

\begin{definition}[Hatted basis]\label{def:hatpsi}
For \(\lambda\neq\varnothing\) and \(s\in\A_\lambda\), define
\begin{equation}\label{eq:hatpsidef}
\widehat\Ps_\lambda^s
:=
\frac{\Ps_\lambda^s}{([s]-1)\varpi_\lambda}.
\end{equation}
Also set \(\widehat\Ps_\varnothing^{(0,0)}:=1\).
\end{definition}
By Proposition~\ref{prop:piinftypsi},
\[
\pi_\infty\widehat\Ps_\lambda^s=z^{-|\lambda|}
\]
for all \(\lambda\).
For $\lambda\ne\varnothing$,
\[
\pi_0\widehat\Ps_\lambda^s=\frac{\widehat J_\lambda}{[s]-1}.
\]

\subsection{Creation and annihilation}\label{ssec:creation}

Multiplication by $z^{-1}$ raises the total degree by one, and the
\emph{down map}
\[
\downmap:\La[z^{-1}]\to\La[z^{-1}]
\]
lowers it by one.  They satisfy $\downmap z^{-1}=\Id$ and
$z^{-1}\downmap=\Id-\pi_0$.

For $\gamma\ne\varnothing$ and $r\in\R_\gamma$, put
\[
\varphi_\gamma^r:=\sum_{a\in\A_\gamma}\frac{\tau_\gamma^a\,[a]}{[a]-[r^+]}\,\Ps_\gamma^a
=\LL\,(\LL-[r^+])^{-1}J_\gamma ,
\]
the inverse taken on $Z_\gamma$, where $[r^+]$ is not an eigenvalue since
$r^+\notin\A_\gamma$.

\begin{lemma}\label{lem:phivectors}
Let $\gamma\ne\varnothing$.
\begin{enumerate}
\item[(i)] For $r\in\R_\gamma$, $\varphi_\gamma^r\in\La^\infty$ and
$\LL^+\varphi_\gamma^r=[r^+]\,\varphi_\gamma^r$.
\item[(ii)] For $r\in\R_\gamma$, $\Q\varphi_\gamma^r=J_\gamma$ and
$\pi_\infty\varphi_\gamma^r=\varpi_\gamma z^{-|\gamma|}$.
\item[(iii)] $\D_\Mq^{-1}\Q^*J_\gamma
=-\sum_{r\in\R_\gamma}[r^+]\,\tilde\tau_\gamma^{\,r^+}\varphi_\gamma^r$.
\end{enumerate}
\end{lemma}
\begin{proof}
Fix $r$ and put $c_a=\tau_\gamma^a[a]/([a]-[r^+])$,
so that $\varphi_\gamma^r=\sum_ac_a\Ps_\gamma^a$.  By \eqref{eq:Tminus1},
$T_\gamma([r^+])=1$, so Lemma~\ref{lem:kerovmoments} gives
$\sum_ac_a=1-T_\gamma([r^+])=0$.  Hence
$\pi_0\varphi_\gamma^r=(\sum_ac_a)J_\gamma=0$ and $\varphi_\gamma^r\in\La^\infty$.

For $F\in\La^\infty$ the block form (Lemma~\ref{lem:blockform}) gives
$\pi_0\LL F=\Q F$ and $\LL^+F=\pi_-\LL F-\D_\Mq^{-1}\Q^*\Q F$; it also gives
$\LL J_\gamma=J_\gamma+\D_\Mq^{-1}\Q^*J_\gamma$.  Take $F=\varphi_\gamma^r$.  Then
$\LL F=\sum_ac_a[a]\Ps_\gamma^a$, so $\Q F=\bigl(\sum_ac_a[a]\bigr)J_\gamma=J_\gamma$, since
\[
\sum_ac_a[a]=\sum_a\tau_\gamma^a[a]+[r^+]\sum_a\frac{\tau_\gamma^a[a]}{[a]-[r^+]}
=1+[r^+]\bigl(1-T_\gamma([r^+])\bigr)=1 .
\]
Hence, by \eqref{eq:jackpsiexp},
\[
\LL^+F=\LL F-J_\gamma-\D_\Mq^{-1}\Q^*J_\gamma=\LL F-\LL J_\gamma
=\sum_a(c_a-\tau_\gamma^a)[a]\,\Ps_\gamma^a=[r^+]\,F,
\]
since $(c_a-\tau_\gamma^a)[a]=[r^+]\,c_a$.  This proves (i) and the first part
of (ii).  For the second part, Proposition~\ref{prop:piinftypsi} gives
$\pi_\infty\varphi_\gamma^r=\varpi_\gamma z^{-|\gamma|}\sum_ac_a([a]-1)$, and
\[
\sum_ac_a([a]-1)=\sum_a\tau_\gamma^a[a]+([r^+]-1)\sum_a\frac{\tau_\gamma^a[a]}{[a]-[r^+]}
=1+([r^+]-1)\bigl(1-T_\gamma([r^+])\bigr)=1 .
\]
For (iii), \eqref{eq:jackpsiexp} and the block form give
$\D_\Mq^{-1}\Q^*J_\gamma=\LL J_\gamma-J_\gamma
=\sum_a\tau_\gamma^a([a]-1)\Ps_\gamma^a$.  On the right of (iii) the coefficient
of $\Ps_\gamma^a$ is
$\tau_\gamma^a[a]\sum_r[r^+]\tilde\tau_\gamma^{\,r^+}/([r^+]-[a])
=\tau_\gamma^a[a](1-1/[a])=\tau_\gamma^a([a]-1)$, by \eqref{eq:pfvalues}
with $a$ in place of $s$.
\end{proof}

\begin{proposition}[compare {\cite[(94)]{Mickler:2023}}]\label{lem:zcreation}
For a partition $\lambda$ and $s\in\A_\lambda$,
\begin{equation}\label{eq:zcreation}
z^{-1}\Ps_\lambda^s
=
\Mq^{-a'(s)}\sum_{r\in\A_{\lambda+s}}
\frac{\tau_{\lambda+s}^r}{1-[s^+]/[r]}\,\Ps_{\lambda+s}^r
=\Mq^{-a'(s)}\varphi_{\lambda+s}^s\in Z_{\lambda+s},
\end{equation}
where $[s^+]=[s]\,\Mt/\Mq$.  In the hatted basis this reads
\begin{equation}\label{eq:hat-zcreation}
z^{-1}\widehat\Ps_\lambda^s
=
\sum_{d\in\A_{\lambda+s}}
\frac{\tau_{\lambda+s}^d([d]-1)}{1-[s^+]/[d]}
\widehat\Ps_{\lambda+s}^{\,d},
\end{equation}
and the coefficients in \eqref{eq:hat-zcreation} sum to $1$.
\end{proposition}
\begin{proof}
By induction on $|\lambda|$, following \cite[Theorem~3.4]{Mickler:2023} with
Lemma~\ref{lem:pi0homo} in place of the derivation property used there, we
show that $z^{-1}\Ps_\lambda^s=\Mq^{-a'(s)}\varphi_{\lambda+s}^s$.

If $\lambda=\varnothing$ then $s=(0,0)$ and $\Ps_\varnothing^{(0,0)}=1$.  The
space $\La^\infty_1$ is spanned by $z^{-1}$, and $\Q z^{-1}=Q_1=(1-\Mt)p_1=J_{(1)}$,
so Lemma~\ref{lem:phivectors}(ii) forces $\varphi_{(1)}^{(0,0)}=z^{-1}$.

Now let $|\lambda|=n\ge1$ and assume the statement for all partitions of size
$n-1$.  By Corollary~\ref{cor:NSspectral}, $\LL^+$ is diagonalizable on
$\La^\infty_{n+1}$, with eigenvectors $z^{-1}\Ps_\mu^a$ ($\mu\vdash n$,
$a\in\A_\mu$) of eigenvalue $[a^+]$.  Let $P$ be the
projection of $\La^\infty_{n+1}$ onto the $[s^+]$-eigenspace of $\LL^+$
along the other eigenspaces, and let $P'$ be the projection of
$\La[z^{-1}]_n$ onto the $[s]$-eigenspace of $\LL$.  Distinct boxes have
distinct contents, so $[a^+]=[s^+]$ only when $a=s$.  Hence $z^{-1}$
maps the $[s]$-eigenspace of $\LL$ onto the $[s^+]$-eigenspace of
$\LL^+$, and the sum of the other eigenspaces of $\LL$ onto the sum of the
other eigenspaces of $\LL^+$, so $Pz^{-1}=z^{-1}P'$.

By the Pieri rule \cite[Chapter~VI, (6.24)(iv) and (6.23)]{Macdonald:1995},
together with $J_\mu=\prod_{b\in\mu}h_\mu(b)\,P_\mu$ \cite[Chapter~VI, (8.3)]{Macdonald:1995},
\[
J_{(1)}J_\lambda=\sum_{a\in\A_\lambda}g^aJ_{\lambda+a},
\qquad g^a\ne0 .
\]
Apply $\LL$.  By Lemma~\ref{lem:pi0homo} the left side becomes
$\LL(J_{(1)})\LL(J_\lambda)$.  The block form gives
$\LL J_\gamma=J_\gamma+\D_\Mq^{-1}\Q^*J_\gamma$ for every $\gamma$, and in
particular $\LL J_{(1)}=J_{(1)}+cz^{-1}$ with $c=(1-\Mt)(1-\Mq)/\Mq$.  Subtracting
$J_{(1)}J_\lambda$ from both sides gives
\begin{equation}\label{eq:creation-pieri}
\sum_{a\in\A_\lambda}g^a\,\D_\Mq^{-1}\Q^*J_{\lambda+a}
=J_{(1)}\,\D_\Mq^{-1}\Q^*J_\lambda+c\,z^{-1}\LL J_\lambda .
\end{equation}
We apply $P$ to the three terms.

By Lemma~\ref{lem:phivectors}(i) and (iii), $\D_\Mq^{-1}\Q^*J_{\lambda+a}$ is a
combination of eigenvectors $\varphi_{\lambda+a}^r$ with eigenvalues
$[r^+]$, $r\in\R_{\lambda+a}$.  Such an eigenvalue equals $[s^+]$ only
if $r=s$, and $s\in\R_{\lambda+a}$ only if $a=s$.  So the left side of
\eqref{eq:creation-pieri} projects to
$-g^s[s^+]\,\tilde\tau_{\lambda+s}^{\,s^+}\varphi_{\lambda+s}^s$, and this
coefficient is nonzero: no factor of $\tilde\tau_{\lambda+s}^{\,s^+}$
vanishes, because $s^+\notin\A_{\lambda+s}$.

By \eqref{eq:jackpsiexp},
$c\,z^{-1}\LL J_\lambda=c\sum_a\tau_\lambda^a[a]\,z^{-1}\Ps_\lambda^a$ projects to
$c\,\tau_\lambda^s[s]\,z^{-1}\Ps_\lambda^s$.

For the remaining term, Lemma~\ref{lem:phivectors}(iii) for $\lambda$ and the
induction hypothesis $\varphi_\lambda^r=\Mq^{a'(r)}z^{-1}\Ps_{\lambda-r}^r$ give
\[
J_{(1)}\,\D_\Mq^{-1}\Q^*J_\lambda
=-\sum_{r\in\R_\lambda}[r^+]\,\tilde\tau_\lambda^{\,r^+}\Mq^{a'(r)}\,
z^{-1}\bigl(J_{(1)}\Ps_{\lambda-r}^r\bigr).
\]
By Lemma~\ref{lem:pi0homo},
$\LL(J_{(1)}\Ps_{\lambda-r}^r)=(J_{(1)}+cz^{-1})[r]\Ps_{\lambda-r}^r$, that is,
\[
(\LL-[r])\bigl(J_{(1)}\Ps_{\lambda-r}^r\bigr)=c[r]\,z^{-1}\Ps_{\lambda-r}^r
=c[r]\,\Mq^{-a'(r)}\varphi_\lambda^r .
\]
Apply $P'$, which commutes with $\LL$.  Since $\varphi_\lambda^r\in Z_\lambda$
and the addable corners of $\lambda$ have distinct contents, $P'$ keeps only
the $\Ps_\lambda^s$-component of $\varphi_\lambda^r$, and the definition of
$\varphi_\lambda^r$ gives
$P'\varphi_\lambda^r=\tau_\lambda^s[s]\,\Ps_\lambda^s/([s]-[r^+])$; since
$[s]\ne[r]$ and $[s]\ne[r^+]$,
\[
P'\bigl(J_{(1)}\Ps_{\lambda-r}^r\bigr)
=\frac{c[r]\,\Mq^{-a'(r)}\tau_\lambda^s[s]}{([s]-[r])([s]-[r^+])}\,\Ps_\lambda^s .
\]
Thus both terms on the right of \eqref{eq:creation-pieri} project to multiples
of $z^{-1}\Ps_\lambda^s$, so $\varphi_{\lambda+s}^s$ is a multiple of
$z^{-1}\Ps_\lambda^s$.  Since $\Q\varphi_{\lambda+s}^s=J_{\lambda+s}\ne0$, the
multiple is nonzero, and $z^{-1}\Ps_\lambda^s$ is a multiple of
$\varphi_{\lambda+s}^s$.  Applying $\pi_\infty$, Proposition~\ref{prop:piinftypsi}
and Lemma~\ref{lem:phivectors}(ii) identify the multiple as
$([s]-1)\varpi_\lambda/\varpi_{\lambda+s}=\Mq^{-a'(s)}$, by
\eqref{eq:varpi-ratio}.

For \eqref{eq:hat-zcreation}, let first $\lambda\ne\varnothing$.  The
coefficient in \eqref{eq:zcreation} is multiplied by
\[
\frac{([d]-1)\varpi_{\lambda+s}}{([s]-1)\varpi_\lambda},
\qquad
\Mq^{-a'(s)}\frac{\varpi_{\lambda+s}}{([s]-1)\varpi_\lambda}=1,
\]
by \eqref{eq:varpi-ratio}.  For
$\lambda=\varnothing$ we have $\widehat\Ps_\varnothing^{(0,0)}=\Ps_\varnothing^{(0,0)}=1$
and $a'(s)=0$, and the multiplier is $([d]-1)\varpi_{(1)}=[d]-1$, which gives the
same formula.  Applying $\pi_\infty$ to \eqref{eq:hat-zcreation} and using
$\pi_\infty\widehat\Ps_\gamma^c=z^{-|\gamma|}$ (Definition~\ref{def:hatpsi})
shows that the coefficients sum to $1$.
\end{proof}

\begin{corollary}[compare {\cite[(117)]{Mickler:2023}}]\label{lem:pidown}
For a partition $\lambda$ and an addable corner $s\in\A_\lambda$,
\begin{equation}\label{eq:pidown}
\downmap\Ps_\lambda^s
=
\sum_{r\in\R_\lambda}
\Mq^{a'(r)}\,[s]\,\frac{\tilde\tau_\lambda^{\,r^+}}{1-[s]/[r^+]}\,\Ps_{\lambda-r}^{\,r}
\in X_\lambda,
\end{equation}
where $X_\lambda:=\operatorname{span}\{\Ps_{\lambda-r}^r:r\in\R_\lambda\}$.  For
$\lambda\ne\varnothing$, in the hatted basis this reads
\begin{equation}\label{eq:hat-pidown}
\downmap\widehat\Ps_\lambda^s
=
\frac{[s]}{[s]-1}\sum_{r\in\R_\lambda}
\frac{\tilde\tau_\lambda^{\,r^+}}{1-[s]/[r^+]}
\widehat\Ps_{\lambda-r}^{\,r},
\end{equation}
and the coefficients in \eqref{eq:hat-pidown} sum to $1$.
\end{corollary}
\begin{proof}
For $\lambda=\varnothing$ both sides vanish, so let $\lambda\ne\varnothing$.
By the definition of $\Ps_\lambda^s$, the vector $\pi_-\Ps_\lambda^s$ is the value
at $u=[s]$ of
\[
(\D_\Mq-\Lc^{NS}/u)^{-1}\Q^*J_\lambda
=u\,(u-\LL^+)^{-1}\D_\Mq^{-1}\Q^*J_\lambda
=-\sum_{r\in\R_\lambda}\frac{u\,[r^+]\,\tilde\tau_\lambda^{\,r^+}}{u-[r^+]}\,
\varphi_\lambda^r,
\]
where the last equality is Lemma~\ref{lem:phivectors}(i) and (iii).  The
right side is regular at $u=[s]$, because $[s]\ne[r^+]$.  By
Proposition~\ref{lem:zcreation} applied to $\lambda-r$ and its addable corner
$r$, $\varphi_\lambda^r=\Mq^{a'(r)}z^{-1}\Ps_{\lambda-r}^r$.  Multiplying by $z$
gives \eqref{eq:pidown}, since $-[s][r^+]/([s]-[r^+])=[s]/(1-[s]/[r^+])$.

For \eqref{eq:hat-pidown}, for $\lambda-r\ne\varnothing$ the coefficient in
\eqref{eq:pidown} is
multiplied by
\[
\frac{([r]-1)\varpi_{\lambda-r}}{([s]-1)\varpi_\lambda},
\qquad
\Mq^{a'(r)}([r]-1)\frac{\varpi_{\lambda-r}}{\varpi_\lambda}=1,
\]
by \eqref{eq:varpi-ratio} for $\lambda-r$ and $r$.  For $\lambda=(1)$ we
have $r=(0,0)$, $\Mq^{a'(r)}=1$ and $\widehat\Ps_\varnothing^{(0,0)}=\Ps_\varnothing^{(0,0)}=1$,
and the multiplier is $1/(([s]-1)\varpi_{(1)})=1/([s]-1)$, which gives the same
formula.  The coefficients in \eqref{eq:hat-pidown} sum to $1$ as in the
proof of Proposition~\ref{lem:zcreation}, since $\pi_\infty$ commutes with
$z$ and $\pi_-$.
\end{proof}

By Proposition~\ref{lem:zcreation}, the eigenbasis of
Corollary~\ref{cor:NSspectral} is
$\{\Mq^{-a'(r)}\varphi_\lambda^r:\lambda\vdash N,\ r\in\R_\lambda\}$.

\section{The product-resolvent formula}\label{sec:proof}

We prove Theorem~\ref{thm:main} by an induction for the $\mathbf{y}$-trace of
the normalized product defect.

\subsection{The $\bT$-factor}

\begin{definition}[$\bT$-factor] For finite sets of boxes $\mu,\nu$, in
particular partitions or single boxes, define their $\bT$ factor as
\begin{equation}\label{def:tfactor}
\bT_{\mu\star\nu}(u)
:=
\prod_{x\in\mu,\ y\in\nu}
\mathcal N\left(u/[x+y]\right).
\end{equation}
It is symmetric, $\bT_{\mu\star\nu}=\bT_{\nu\star\mu}$, and multiplicative in
each argument over disjoint unions of boxes.
\end{definition}

\subsection{Preliminaries}\label{ssec:resolvent-facts}

For $F\in\La[z^{-1}]$ homogeneous of total degree $N$ we define the
normalized $\mathbf{y}$-trace by
\begin{equation}\label{eq:yudef}
\mathbf{y}_u(F):=z^{N}\,\pi_\infty(1-\LL/u)^{-1}F
=u\,z^{N}\,\pi_\infty(u-\LL)^{-1}F\ \in\ \QQ(\Mq,\Mt,u).
\end{equation}
Because $\LL$ preserves each total degree, $\mathbf{y}_u$ is a
$\QQ(\Mq,\Mt)$-linear functional on $\La[z^{-1}]_N$.
Since $\pi_\infty$ is multiplicative
and $\pi_\infty\widehat\Ps_\gamma^c=z^{-|\gamma|}$
(Definition~\ref{def:hatpsi}), Theorem~\ref{thm:main} reads
\begin{equation}\label{eq:yutarget}
\mathbf{y}_u(\widehat\Ps_\mu^s\widehat\Ps_\nu^t)
=\frac{\bT_{\mu\star\nu}(u)}{1-[s][t]/u}.
\end{equation}

\begin{remark}[Cancellation of the corner factor]\label{rem:corner-cancellation}
The denominator in \eqref{eq:yutarget} need not produce a pole.  For
$\mu=\nu=(1)$, $s=(0,1)$ and $t=(1,0)$, its right-hand side is
\[
\frac{\mathcal N(u)}{1-\Mt/(\Mq u)}
=\frac{u(u-1)}{(u-\Mq^{-1})(u-\Mt)},
\]
which is regular at $u=[s+t]=\Mt/\Mq$.
\end{remark}

\subsection{The normalized product defect}\label{ssec:defect}

\begin{definition}\label{def:beta}
For $f,g\in\La[z^{-1}]$, define the \emph{normalized product defect}
\begin{equation}\label{eq:beta-def}
\beta(f,g):=\LL^{-1}\bigl(\LL(fg)-\LL(f)\LL(g)\bigr)
=fg-\LL^{-1}\bigl(\LL(f)\LL(g)\bigr).
\end{equation}
The inverse exists degree by degree, since the eigenvalues of $\LL$ are
nonzero corner contents.
\end{definition}

\begin{lemma}[Polynomial form of the defect]\label{lem:beta-structure}
For $f,g\in\La[z^{-1}]$,
\begin{align}
\beta(f,g)
&=fg-\mathbf Q^{-1}\bigl((\mathbf Qf)(\mathbf Qg)\bigr)
\label{eq:beta-Q}\\
&=\mathbf Q(fg)-(\mathbf Qf)g-f(\mathbf Qg)+fg.
\label{eq:beta-coboundary}
\end{align}
In particular, $\beta$ is symmetric and $\La$-bilinear, vanishes if either
argument belongs to $\La$, and satisfies $\pi_\infty\beta(f,g)=0$.
\end{lemma}
\begin{proof}
The operator $\mathbf Q$ is $\La$-linear, and $\mathbf Q-\Id$ strictly
lowers the degree in $z^{-1}$, so $\mathbf Q$ has a $\La$-linear inverse.
Since $p_n^\perp=n\,\partial/\partial p_n$, the exponent defining
$\mathbf Q^\dagger$ is a derivation that lowers the ordinary polynomial
degree in the power sums.  Its exponential is therefore locally finite,
and Taylor's formula gives the simultaneous substitution
\[
\mathbf Q^\dagger F
=\left.F\right|_{p_n\mapsto p_n+(\Mq^{-n}-1)z^{-n}\ (n\ge1)},
\qquad F\in\La[z^{-1}],
\]
with $z^{-1}$ fixed.  This substitution preserves products, and reversing
the translations gives its inverse.  Thus $\mathbf Q^\dagger$ is a
$\QQ(\Mq,\Mt)$-algebra automorphism.  The operator $\D_{\Mq^{-1}\Mt}$ is also
an algebra automorphism, since it rescales $z^{-1}$.  The factorization
\eqref{eq:Laxdef} gives
\[
\LL^{-1}\bigl(\LL(f)\LL(g)\bigr)
=\mathbf Q^{-1}\bigl((\mathbf Qf)(\mathbf Qg)\bigr),
\]
which proves \eqref{eq:beta-Q}.  This expression is $\La$-bilinear and
vanishes when either argument is in $\La$, because $\mathbf Q$ fixes $\La$.

To prove \eqref{eq:beta-coboundary}, it therefore suffices to consider
$f=z^{-m}$, $g=z^{-n}$.  The identity
\[
\mathbf Qz^{-m-1}=z^{-1}\mathbf Qz^{-m}+Q_{m+1}
\]
and \eqref{eq:beta-Q} give, for $m\ge0$, $n\ge1$,
\[
\beta(z^{-m},z^{-n})-\beta(z^{-m-1},z^{-n+1})
=Q_{m+1}z^{-n+1}-Q_nz^{-m}.
\]
Iterating down to $\beta(z^{-m-n},1)=0$ yields
\begin{equation}\label{eq:beta-monomials}
\beta(z^{-m},z^{-n})
=\sum_{j=0}^{n-1}
\left(Q_{m+j+1}z^{-(n-j-1)}-Q_{n-j}z^{-(m+j)}\right).
\end{equation}
Comparing this with
$\mathbf Qz^{-k}=\sum_{i=0}^kQ_i z^{-(k-i)}$ proves
\eqref{eq:beta-coboundary} on monomials, hence in general by
$\La$-bilinearity.  Finally, $\pi_\infty Q_i=0$ for $i\ge1$, so
\eqref{eq:beta-monomials} also gives $\pi_\infty\beta(f,g)=0$.
\end{proof}

\begin{lemma}[Defect recursion]\label{lem:beta-recursion}
For $f,g\in\La[z^{-1}]$,
\begin{equation}\label{eq:beta-recursion}
\beta(f,g)=\beta(z\,\pi_- f,z^{-1}g)
+(z\,\pi_- f)\Q(z^{-1}g)-\Q(\pi_-f)g.
\end{equation}
\end{lemma}
\begin{proof}
Equation~\eqref{eq:beta-coboundary} identifies $\beta$ as the derivation
defect of $\mathbf Q-\Id$.  Expanding it gives the Hochschild cocycle identity
\[
\beta(fg,h)-\beta(f,gh)=f\beta(g,h)-\beta(f,g)h.
\]
Also, $\mathbf Q(z^{-1}g)=z^{-1}\mathbf Qg+\Q(z^{-1}g)$, so
\[
\beta(z^{-1},g)=\Q(z^{-1}g)-Q_1g.
\]
Apply the cocycle identity to $(z\,\pi_- f,z^{-1},g)$, using
$z^{-1}(z\,\pi_-f)=\pi_-f$ and $\beta(\pi_0f,g)=0$.  The two terms containing
$Q_1$ cancel, leaving \eqref{eq:beta-recursion}.
\end{proof}

\begin{lemma}[Universal trace identities]\label{lem:universal-traces}
For $f,g\in\La[z^{-1}]$,
\begin{equation}\label{eq:universal-traces}
\Q\bigl(z^{-1}\beta(f,g)\bigr)=0,
\qquad
\Q\bigl(\pi_-\beta(f,g)\bigr)=-\Q(\pi_-f)\Q(\pi_-g).
\end{equation}
\end{lemma}
\begin{proof}
By $\La$-bilinearity and the vanishing of $\beta$ on $\La$, it suffices to
take $f=z^{-m}$, $g=z^{-n}$ with $m,n\ge1$.  Apply $\Q(z^{-1}\,\cdot)$ to
\eqref{eq:beta-monomials}.  Since $\Q(z^{-k})=Q_k$ for $k\ge1$, each pair
of summands cancels.  Applying $\Q\pi_-$ instead gives
\[
\begin{aligned}
\Q\bigl(\pi_-\beta(z^{-m},z^{-n})\bigr)
&=\sum_{j=0}^{n-2}Q_{m+j+1}Q_{n-j-1}
-\sum_{j=0}^{n-1}Q_{n-j}Q_{m+j}\\
&=-Q_mQ_n,
\end{aligned}
\]
where the first sum is empty when $n=1$.  This proves the second identity.
\end{proof}

\subsection{The inductive proof}\label{ssec:mainproof}

\begin{lemma}[Projection identities]\label{lem:hat-projections}
For $\lambda\ne\varnothing$ and $s\in\A_\lambda$,
\begin{equation}\label{eq:hat-projections}
\Q(\pi_-\widehat\Ps_\lambda^s)=\widehat J_\lambda,
\qquad
\Q(z^{-1}\widehat\Ps_\lambda^s)=\widehat J_{\lambda+s}.
\end{equation}
\end{lemma}
\begin{proof}
By the block form and \eqref{eq:hatpsidef},
\[
\Q(\pi_-\widehat\Ps_\lambda^s)
=\pi_0(\LL-1)\widehat\Ps_\lambda^s
=([s]-1)\pi_0\widehat\Ps_\lambda^s=\widehat J_\lambda.
\]
By Proposition~\ref{lem:zcreation} and Lemma~\ref{lem:phivectors}(ii),
$\Q(z^{-1}\Ps_\lambda^s)=\Mq^{-a'(s)}J_{\lambda+s}$.
Dividing by $([s]-1)\varpi_\lambda$ and using \eqref{eq:varpi-ratio} gives
the second identity.
\end{proof}

For $\mu\ne\varnothing$, put $\widehat\tau_\mu^s=([s]-1)\tau_\mu^s$.  Then
\begin{equation}\label{eq:hat-kerov-moments}
\sum_{s\in\A_\mu}\widehat\tau_\mu^s\widehat\Ps_\mu^s=\widehat J_\mu,
\end{equation}
which is \eqref{eq:jackpsiexp} in the hatted normalization.

For a single box $r$ and any partition $\gamma$, the relation $[r+y]=[r][y]$
gives $\bT_{r\star\gamma}(u)=S_\gamma(u/[r])$, so the definition of $T_\gamma$
in Proposition~\ref{cor:transfer-hierarchy} gives
\begin{equation}\label{eq:boxstar-single}
\bT_{r\star\gamma}(u)
=\frac{u-[r]}{[r]}\left(T_\gamma(u/[r])-1\right).
\end{equation}
Splitting the product over $y\in\nu+t$ in the definition of
$\bT_{\mu\star(\nu+t)}$ gives
\begin{equation}\label{eq:boxmult}
\bT_{\mu\star(\nu+t)}(u)
=\bT_{\mu\star\nu}(u)\,\bT_{\mu\star t}(u).
\end{equation}
For $\gamma\ne\varnothing$ and any single box $r$, the hatted Kerov sum is
\begin{equation}\label{eq:hat-kerov-kernel}
\sum_{a\in\A_\gamma}\frac{\widehat\tau_\gamma^a}{1-[r][a]/u}
=\bT_{r\star\gamma}(u)-1.
\end{equation}
Indeed, Definition~\ref{def:kerov} and $\sum_a\tau_\gamma^a=1$
(Lemma~\ref{lem:kerovmoments}) give the sum as
$(u/[r]-1)T_\gamma(u/[r])-u/[r]$, which equals the right-hand side by
\eqref{eq:boxstar-single}.

For partitions $\mu,\nu$ and addable corners $s\in\A_\mu$, $t\in\A_\nu$, write
\[
\beta_{\mu,\nu}^{s,t}:=\beta(\widehat\Ps_\mu^s,\widehat\Ps_\nu^t).
\]

\begin{lemma}[Product and defect traces]\label{lem:y-product-defect}
For all such $\mu,\nu,s,t$,
\begin{equation}\label{eq:y-product-defect}
\mathbf{y}_u(\beta_{\mu,\nu}^{s,t})
=(1-[s][t]/u)\,\mathbf{y}_u(\widehat\Ps_\mu^s\widehat\Ps_\nu^t)-1.
\end{equation}
\end{lemma}
\begin{proof}
Put $F=\widehat\Ps_\mu^s\widehat\Ps_\nu^t$, of total degree
$N=|\mu|+|\nu|$.  Definition~\ref{def:beta} and the eigenvalue equations give
$\LL\beta_{\mu,\nu}^{s,t}=(\LL-[s][t])F$.  Apply
$z^N\pi_\infty(u-\LL)^{-1}$ to this equality.  Using
$(u-\LL)^{-1}\LL=u(u-\LL)^{-1}-\Id$,
$\pi_\infty\beta_{\mu,\nu}^{s,t}=0$ from Lemma~\ref{lem:beta-structure},
and $\pi_\infty F=z^{-N}$ gives \eqref{eq:y-product-defect}.
\end{proof}

\begin{proposition}[$\mathbf{y}$-trace formula]\label{prop:ytrace}
For all partitions $\mu,\nu$ and addable corners $s\in\A_\mu$, $t\in\A_\nu$,
\begin{equation}\label{eq:y-beta}
\mathbf{y}_u(\beta_{\mu,\nu}^{s,t})=\bT_{\mu\star\nu}(u)-1.
\end{equation}
\end{proposition}
\begin{proof}
We follow the Jack induction of \cite[Theorem~4.11]{Mickler:2023}, at fixed
total size $N=|\mu|+|\nu|$, inducting on the size of the smaller partition.
If either partition is empty, one argument of $\beta$ is $1$, so both sides
of \eqref{eq:y-beta} vanish.  By symmetry, assume
$0<|\mu|\le|\nu|$ and fix $t\in\A_\nu$.  For every $r\in\R_\mu$, the pair
$(\mu-r,\nu+t)$ has total size $N$ and smaller size $|\mu|-1$.
The induction hypothesis applies to this pair, and
Lemma~\ref{lem:y-product-defect} also gives its product formula
\eqref{eq:yutarget}.

By \eqref{eq:hat-pidown}, write
\[
z\,\pi_-\widehat\Ps_\mu^s
=\sum_{r\in\R_\mu}c_r^s\widehat\Ps_{\mu-r}^{\,r},
\qquad \sum_{r\in\R_\mu}c_r^s=1.
\]
Likewise the coefficients of $z^{-1}\widehat\Ps_\nu^t$ in the hatted basis
of $Z_{\nu+t}$ sum to $1$, by \eqref{eq:hat-zcreation}.
Lemmas~\ref{lem:beta-recursion}
and~\ref{lem:hat-projections} give
\begin{equation}\label{eq:beta-induction}
\beta_{\mu,\nu}^{s,t}
=\sum_{r\in\R_\mu}c_r^s\left(
\beta(\widehat\Ps_{\mu-r}^{\,r},z^{-1}\widehat\Ps_\nu^t)
+\widehat\Ps_{\mu-r}^{\,r}\widehat J_{\nu+t}\right)
-\widehat J_\mu\widehat\Ps_\nu^t.
\end{equation}
Expanding $z^{-1}\widehat\Ps_\nu^t$ by the creation formula and applying the
induction hypothesis gives
\[
\mathbf{y}_u\bigl(\beta(\widehat\Ps_{\mu-r}^{\,r},z^{-1}\widehat\Ps_\nu^t)\bigr)
=\bT_{(\mu-r)\star(\nu+t)}(u)-1.
\]
For the other term in the sum, expand $\widehat J_{\nu+t}$ using
\eqref{eq:hat-kerov-moments}.  The product formula for the smaller pair and
\eqref{eq:hat-kerov-kernel} give
\[
\begin{aligned}
\mathbf{y}_u(\widehat\Ps_{\mu-r}^{\,r}\widehat J_{\nu+t})
&=\bT_{(\mu-r)\star(\nu+t)}(u)
\sum_{d\in\A_{\nu+t}}\frac{\widehat\tau_{\nu+t}^d}{1-[r][d]/u}\\
&=\bT_{\mu\star(\nu+t)}(u)-\bT_{(\mu-r)\star(\nu+t)}(u).
\end{aligned}
\]
Thus the two terms combine independently of $r$.  Put
$H:=\mathbf{y}_u(\widehat J_\mu\widehat\Ps_\nu^t)$.  Taking the trace of
\eqref{eq:beta-induction} and using $\sum_r c_r^s=1$ yields
\begin{equation}\label{eq:beta-trace-step}
\mathbf{y}_u(\beta_{\mu,\nu}^{s,t})
=\bT_{\mu\star(\nu+t)}(u)-1-H.
\end{equation}
The right-hand side is independent of $s$.  Using this independence,
\eqref{eq:hat-kerov-moments}, Lemma~\ref{lem:y-product-defect}, and
\eqref{eq:hat-kerov-kernel}, we obtain
\[
\begin{aligned}
H&=\sum_{a\in\A_\mu}\widehat\tau_\mu^a\,
\mathbf{y}_u(\widehat\Ps_\mu^a\widehat\Ps_\nu^t)\\
&=\bigl(1+\mathbf{y}_u(\beta_{\mu,\nu}^{s,t})\bigr)
\sum_{a\in\A_\mu}\frac{\widehat\tau_\mu^a}{1-[a][t]/u}\\
&=\bigl(1+\mathbf{y}_u(\beta_{\mu,\nu}^{s,t})\bigr)
\bigl(\bT_{\mu\star t}(u)-1\bigr).
\end{aligned}
\]
Substituting $H$ into \eqref{eq:beta-trace-step} and using \eqref{eq:boxmult}
gives
\[
\bigl(1+\mathbf{y}_u(\beta_{\mu,\nu}^{s,t})\bigr)\bT_{\mu\star t}(u)
=\bT_{\mu\star(\nu+t)}(u)
=\bT_{\mu\star\nu}(u)\bT_{\mu\star t}(u).
\]
Cancelling the nonzero rational factor $\bT_{\mu\star t}(u)$ proves
\eqref{eq:y-beta} and completes the induction.
\end{proof}

\begin{proof}[Proof of Theorem~\ref{thm:main}]
Proposition~\ref{prop:ytrace} and Lemma~\ref{lem:y-product-defect} give
\eqref{eq:yutarget}, which is equivalent to \eqref{eq:maintheorem} by the
hatted normalization.
\end{proof}

\begin{corollary}\label{cor:defect}
For $s\in\A_\mu$ and $t\in\A_\nu$,
\[
\pi_\infty(1-\LL/u)^{-1}\beta(\Ps_\mu^s,\Ps_\nu^t)
=\bigl(\bT_{\mu\star\nu}(u)-1\bigr)
\pi_\infty(\Ps_\mu^s\Ps_\nu^t).
\]
\end{corollary}
\begin{proof}
Rescale Proposition~\ref{prop:ytrace} using Definition~\ref{def:hatpsi}
and the multiplicativity of $\pi_\infty$.
\end{proof}

\section{The box-sum identity}\label{sec:boxsum}

We now turn to Littlewood--Richardson coefficients.  In the Jack case,
\cite{Mickler:2023} proves the following box-sum identity.

\begin{theorem}[Jack box-sum identity, {\cite[Proposition~5.6]{Mickler:2023}}]
\label{thm:jackboxsummain}
In the conventions of \cite{Mickler:2023}, the content of a box
$s=(s_1,s_2)$ is $[s]=[s_1,s_2]:=s_1\varepsilon_1+s_2\varepsilon_2$, and
$\alpha=-\varepsilon_2/\varepsilon_1$.  Let $\mu,\nu$ be nonempty partitions
and let $J_\mu J_\nu=\sum_{\lambda}g_{\mu\nu}^{\lambda}(\alpha)J_\lambda$.
Then
\begin{equation}\label{eq:jackboxsummain}
\sum_{\lambda}
\widehat g_{\mu\nu}^{\lambda}
\sum_{b\in\lambda}\frac{1}{u-[b]}
=T^{\mathrm{Jack}}_{\mu\star\nu}(u)-1,
\end{equation}
where $\widehat g_{\mu\nu}^{\lambda}=g_{\mu\nu}^{\lambda}(\alpha)
\varpi_\lambda/(\varpi_\mu\varpi_\nu)$ with
$\varpi_\sigma=\prod_{b\in\sigma,\,b\neq(0,0)}[b]$, and
\[
T^{\mathrm{Jack}}_{\mu\star\nu}(u)
=\prod_{x\in\mu,\,y\in\nu}\mathcal N^{\mathrm{Jack}}(u-[x+y]),
\qquad
\mathcal N^{\mathrm{Jack}}(u)=\frac{(u-[0,0])(u-[1,1])}{(u-[1,0])(u-[0,1])}.
\]
\end{theorem}
The hypothesis $\mu,\nu\ne\varnothing$, implicit in
\cite[Proposition~5.6]{Mickler:2023}, is needed: for $\mu=\varnothing$ the left
side is $\sum_{b\in\nu}1/(u-[b])$ and the right side is $0$.

For nonempty $\mu,\nu$, write the Macdonald product and its normalization as
\[
J_\mu J_\nu=\sum_{\lambda}g_{\mu\nu}^{\lambda}J_\lambda,
\qquad
\widehat g_{\mu\nu}^{\lambda}
:=g_{\mu\nu}^{\lambda}\frac{\varpi_\lambda}{\varpi_\mu\varpi_\nu}.
\]
\begin{theorem}[Macdonald box-sum identity]\label{thm:boxsummain}
Let $\mu,\nu$ be nonempty partitions.  Then
\begin{equation}\label{eq:boxsummain}
\sum_{\lambda}
\widehat g_{\mu\nu}^{\lambda}
\sum_{b\in\lambda}\frac{1}{1-[b]/u}
=\bT_{\mu\star\nu}(u)-1.
\end{equation}
\end{theorem}

We write
\begin{equation}\label{eq:Deltau}
\Delta_u(\eta):=\sum_{b\in\eta}\frac{1}{1-[b]/u}
\end{equation}
for the inner sum.  We prove Theorem~\ref{thm:boxsummain} from
Theorem~\ref{thm:main} by a Macdonald analogue of the Jack trace argument
\cite[Theorems~4.27 and~4.28]{Mickler:2023}.

\subsection{The three traces}\label{ssec:traces}

For $G\in\La$ write $[\widehat J_\lambda]G$ for the
coefficient of $\widehat J_\lambda$ in $G$.  On $\La[z^{-1}]_N$, $N\ge1$, we
use $\mathbf{y}_u$ from \eqref{eq:yudef} together with
\[
\mathbf{x}_\gamma(F):=[\widehat J_\gamma]\,\Q(z^{-1}F)\quad(\gamma\vdash N+1),
\qquad
\mathbf{z}_\lambda(F):=[\widehat J_\lambda]\,\Q(\pi_-F)\quad(\lambda\vdash N).
\]
These are the Macdonald analogues of the $\mathbf{x}$-, $\mathbf{y}$- and
$\mathbf{z}$-traces of \cite{Mickler:2023}.  By Lemma~\ref{lem:blockform},
$\Q(\pi_-F)=\pi_0(\LL-1)F$.

\begin{lemma}\label{lem:tracebasis}
For $|\sigma|=N\ge1$ and $c\in\A_\sigma$,
\[
\mathbf{x}_\gamma(\widehat\Ps_\sigma^c)=\delta_{\gamma,\sigma+c},
\qquad
\mathbf{y}_u(\widehat\Ps_\sigma^c)=\frac{1}{1-[c]/u},
\qquad
\mathbf{z}_\lambda(\widehat\Ps_\sigma^c)=\delta_{\lambda\sigma}.
\]
\end{lemma}
\begin{proof}
The first and third identities follow from Lemma~\ref{lem:hat-projections}
and the definitions of $\mathbf{x}_\gamma$ and $\mathbf{z}_\lambda$.
The second follows from $\LL\widehat\Ps_\sigma^c=[c]\widehat\Ps_\sigma^c$ and
$\pi_\infty\widehat\Ps_\sigma^c=z^{-N}$.
\end{proof}

Thus, in the hatted eigenbasis, $\mathbf{x}_\gamma(F)$ is the sum of the coordinates
of $F$ along the basis $\{\widehat\Ps_{\gamma-r}^{\,r}:r\in\R_\gamma\}$ of the space
$X_\gamma$ of Corollary~\ref{lem:pidown}, and $\mathbf{z}_\lambda(F)$ is the sum of
its coordinates along the basis $\{\widehat\Ps_\lambda^s:s\in\A_\lambda\}$ of
$Z_\lambda$.  The down map sends $Z_\gamma$ into $X_\gamma$, by
\eqref{eq:pidown}.

The $\mathbf{y}$-trace of $\beta_{\mu,\nu}^{s,t}$ is given by
Proposition~\ref{prop:ytrace}.  The other two traces follow from the universal
identities in Lemma~\ref{lem:universal-traces}.

\begin{proposition}[The $\mathbf{x}$- and $\mathbf{z}$-traces]\label{prop:traces}
For nonempty $\mu,\nu$ and addable corners $s\in\A_\mu$, $t\in\A_\nu$, put
$N=|\mu|+|\nu|$.  Then
\begin{equation}\label{eq:tripletrace}
\mathbf{x}_\gamma\bigl(\beta_{\mu,\nu}^{s,t}\bigr)=0
\quad(\gamma\vdash N+1),
\qquad
\mathbf{z}_\lambda\bigl(\beta_{\mu,\nu}^{s,t}\bigr)=-\widehat g_{\mu\nu}^{\lambda}
\quad(\lambda\vdash N).
\end{equation}
Equivalently, since
$\sum_\lambda\widehat g_{\mu\nu}^{\lambda}\widehat J_\lambda=\widehat J_\mu\widehat J_\nu$,
\[
\Q\bigl(z^{-1}\beta_{\mu,\nu}^{s,t}\bigr)=0,
\qquad
\Q\bigl(\pi_-\beta_{\mu,\nu}^{s,t}\bigr)=-\widehat J_\mu\widehat J_\nu .
\]
\end{proposition}
\begin{proof}
By Lemma~\ref{lem:hat-projections},
$\Q(\pi_-\widehat\Ps_\mu^s)=\widehat J_\mu$ and
$\Q(\pi_-\widehat\Ps_\nu^t)=\widehat J_\nu$.  Apply
Lemma~\ref{lem:universal-traces} to these two eigenvectors and extract the
coefficients of the hatted Macdonald functions.
\end{proof}

Together with Proposition~\ref{prop:ytrace}, these formulas determine all
three traces of $\beta_{\mu,\nu}^{s,t}$ independently of the corners $s,t$.
They give the Macdonald analogue of the Jack trace formula
\cite[Theorem~4.27]{Mickler:2023}, using the normalized defect
\eqref{eq:beta-def}.

\begin{proposition}[Telescoping trace relation]\label{prop:telescope}
For every $f\in\La[z^{-1}]_N$, $N\ge1$,
\begin{equation}\label{eq:telescope}
\mathbf{y}_u(f)
=\sum_{\gamma\vdash N+1}\mathbf{x}_\gamma(f)\,\Delta_u(\gamma)
-\sum_{\lambda\vdash N}\mathbf{z}_\lambda(f)\,\Delta_u(\lambda).
\end{equation}
\end{proposition}
\begin{proof}
Both sides are linear in $f$, and the vectors $\widehat\Ps_\gamma^c$
($|\gamma|=N$, $c\in\A_\gamma$) form a basis of $\La[z^{-1}]_N$.  By
Lemma~\ref{lem:tracebasis}, the right-hand side on this vector is
\[
\Delta_u(\gamma+c)-\Delta_u(\gamma)
=\frac{1}{1-[c]/u}
=\mathbf{y}_u(\widehat\Ps_\gamma^c),
\]
which proves \eqref{eq:telescope}.
\end{proof}

\begin{proof}[Proof of Theorem~\ref{thm:boxsummain}]
Choose $s\in\A_\mu$, $t\in\A_\nu$ and substitute
$f=\beta_{\mu,\nu}^{s,t}$ in \eqref{eq:telescope}.  By
Propositions~\ref{prop:ytrace} and~\ref{prop:traces},
\[
\bT_{\mu\star\nu}(u)-1
=-\sum_\lambda \mathbf{z}_\lambda(\beta_{\mu,\nu}^{s,t})\Delta_u(\lambda)
=\sum_\lambda\widehat g_{\mu\nu}^{\lambda}\Delta_u(\lambda).
\]
\end{proof}

As $u\to\infty$, $\Delta_u(\lambda)\to|\lambda|=|\mu|+|\nu|$ and
$\bT_{\mu\star\nu}(u)\to1$, so \eqref{eq:boxsummain} gives, for nonempty
$\mu,\nu$,
\begin{equation}\label{eq:ghat-sum}
\sum_\lambda\widehat g_{\mu\nu}^{\lambda}=0.
\end{equation}

\section{A two-parameter strengthening}\label{sec:twoparam}

Having proved the box-sum identity by Lax-operator methods, we now strengthen
it to the two-parameter identity of Theorem~\ref{thm:maintwo}.

\subsection{The second parameter}\label{ssec:secondparam}

The linear map
$\kappa_{u,v}:\La\to\QQ(\Mq,\Mt,u,v)$ is defined on the integral Macdonald basis
by
\begin{equation}\label{eq:kappa-AFLT}
\kappa_{u,v}(J_\lambda)
:=\frac{\epsilon_u(J_\lambda)\epsilon_v(J_\lambda)}
        {\epsilon_{uv}(J_\lambda)}
=\Mt^{n(\lambda)}\prod_{b\in\lambda}
\frac{(1-u/[b])(1-v/[b])}{1-uv/[b]},
\end{equation}
where the product form follows from Theorem~\ref{thm:principal-spec}.

Theorem~\ref{thm:maintwo} controls the multiplicative defect of $\kappa_{u,v}$
on two basis elements through the same box-star factor as in
Theorem~\ref{thm:main}.  Expanding the product of
Macdonald functions in \eqref{eq:kappa-multiplicativity} gives
\begin{equation}\label{eq:kappa-expanded}
\sum_\lambda g_{\mu\nu}^{\lambda}(\Mq,\Mt)\,
\kappa_{u,v}(J_\lambda)
=
\kappa_{u,v}(J_\mu)\kappa_{u,v}(J_\nu)
\bT_{\mu\star\nu}(uv).
\end{equation}

Since $\kappa_{0,v}=\epsilon_v$, for $\lambda\ne\varnothing$ the
definition of $\varpi_\lambda$ gives
$\kappa_{0,v}(J_\lambda)=\varpi_\lambda(1-v)+O\bigl((1-v)^2\bigr)$, while
\[
\frac{\kappa_{u,v}(J_\lambda)}{\kappa_{0,v}(J_\lambda)}
=\prod_{b\in\lambda}\frac{1-u/[b]}{1-uv/[b]}
=1+(1-v)\,\Delta_u(\lambda)+O\bigl((1-v)^2\bigr).
\]
Hence, for nonempty $\mu,\nu$,
\begin{equation}\label{eq:kappa-v1-expansion}
\sum_\lambda g_{\mu\nu}^{\lambda}
\bigl(\kappa_{u,v}(J_\lambda)-\kappa_{0,v}(J_\lambda)\bigr)
=(1-v)^2\varpi_\mu\varpi_\nu
\sum_\lambda\widehat g_{\mu\nu}^{\lambda}\Delta_u(\lambda)
+O\bigl((1-v)^3\bigr),
\end{equation}
so the left-hand side of \eqref{eq:boxsummain} is the coefficient of
$(1-v)^2$ divided by $\varpi_\mu\varpi_\nu$.  On the Lax side, $v=1$ is the
specialization that normalizes the hatted basis:
$\pi_\infty$ restricts to $\epsilon_1$ on $\La$, and
Proposition~\ref{prop:piinftypsi} is the limit $v\to1$ of
Proposition~\ref{prop:epsPsi}.

\begin{example}[The case $\mu=\nu=(1)$]\label{ex:kappa-one-one}
The integral-form product is
\begin{equation}\label{eq:Jone-square}
J_{(1)} \cdot J_{(1)}
=\frac{1-\Mt}{1-\Mq\Mt}J_{(2)}
 +\frac{1-\Mq}{1-\Mq\Mt}J_{(1,1)}.
\end{equation}
Thus the only nonzero coefficients are
\[
g_{(1),(1)}^{(2)}=\frac{1-\Mt}{1-\Mq\Mt},
\qquad
g_{(1),(1)}^{(1,1)}=\frac{1-\Mq}{1-\Mq\Mt}.
\]
The box contents and $n(\lambda)$ are
\[
\begin{array}{c|c|c}
\lambda&\{[b]:b\in\lambda\}&n(\lambda)\\ \hline
(1)&\{1\}&0\\
(2)&\{1,\Mq^{-1}\}&0\\
(1,1)&\{1,\Mt\}&1.
\end{array}
\]
Put
\[
K:=\kappa_{u,v}(J_{(1)})=\frac{(1-u)(1-v)}{1-uv}.
\]
Substitution of the preceding contents into \eqref{eq:kappa-AFLT} gives
\begin{align*}
\kappa_{u,v}(J_{(2)})
&=K\frac{(1-\Mq u)(1-\Mq v)}{1-\Mq uv},\\
\kappa_{u,v}(J_{(1,1)})
&=\Mt K\frac{(1-u/\Mt)(1-v/\Mt)}{1-uv/\Mt}.
\end{align*}
The only pair of boxes is $x=y=(0,0)$, with $[x+y]=1$.  Consequently
\[
\bT_{(1)\star(1)}(uv)=\mathcal N(uv)
=\frac{(1-\Mq uv/\Mt)(1-uv)}{(1-\Mq uv)(1-uv/\Mt)}.
\]
Dividing \eqref{eq:kappa-multiplicativity} by $K$ and multiplying by
$\Mt(1-\Mq\Mt)(1-\Mq uv)(1-uv/\Mt)$ reduces it to the polynomial identity
\begin{align*}
&\,\,(1-\Mt)(1-\Mq u)(1-\Mq v)(\Mt-uv)\\
&+(1-\Mq)(\Mt-u)(\Mt-v)(1-\Mq uv)\\
&=(1-\Mq\Mt)(1-u)(1-v)(\Mt-\Mq uv),
\end{align*}
which follows by direct expansion.
\end{example}

\subsection{Cauchy-kernel form}\label{ssec:whitproof}

The box-star factor also admits a Cauchy-kernel form, giving a plethystic
description of its dependence on $uv$ in Theorem~\ref{thm:maintwo}.  We use
the biexponent generator of Garsia--Haiman--Tesler
\cite[(I.3)]{GarsiaHaimanTesler:1999},
\[
B_\lambda(\Mq,\Mt):=\sum_{b\in\lambda}\Mq^{a'(b)}\Mt^{\ell'(b)}
=\sum_{i\ge1}\Mt^{i-1}\,\frac{1-\Mq^{\lambda_i}}{1-\Mq}.
\]
Since $[b]=\Mt^{\ell'(b)}\Mq^{-a'(b)}$,
$B_\lambda(\Mq,\Mt^{-1})=\sum_{b\in\lambda}[b]^{-1}$.  Put
\[
\overline{\mathfrak s}_\lambda:=-\Mt^{-1}(1-\Mq)B_\lambda(\Mq,\Mt^{-1})
=\sum_{i\ge1}\Mt^{-i}\bigl(\Mq^{\lambda_i}-1\bigr),
\]
which is $\varepsilon_\lambda-\varepsilon_\varnothing$ for
$\varepsilon_\lambda:=\sum_{i\ge1}\Mt^{-i}\Mq^{\lambda_i}$, the notation of
Bourgine--Cassia--Stoyan \cite{BourgineCassiaStoyan:2025} with
$(q_1,q_2)=(\Mt^{-1},\Mq)$.  We use plethystic substitution as in
\cite{GarsiaHaimanTesler:1999}: for a Laurent polynomial $A$ in
$\Mq,\Mt,u,v$ with integer coefficients, $p_k[A]$ is obtained from $A$ by
replacing each of $\Mq,\Mt,u,v$ by its $k$th power.  For such $A,B$ put
\begin{equation}\label{eq:macrkernel}
\Pi(A,B):=\exp\Bigl(\sum_{k\ge1}\frac1k\frac{1-\Mt^k}{1-\Mq^k}\,
p_k[A]\,p_k[B]\Bigr).
\end{equation}
With $p_k(x_1,x_2,\ldots)$ and $p_k(y_1,y_2,\ldots)$ in place of $p_k[A]$
and $p_k[B]$, the right side of \eqref{eq:macrkernel} is Macdonald's Cauchy
kernel $\prod_{i,j}(\Mt x_iy_j;\Mq)_\infty/(x_iy_j;\Mq)_\infty$, where
$(a;\Mq)_\infty=\prod_{r\ge0}(1-a\Mq^r)$ \cite[Chapter~VI, (2.5)
and~(2.6)]{Macdonald:1995}.

\begin{lemma}
\label{lem:TasExp}
As a formal power series at $u=0$,
\begin{equation}\label{eq:TasExp}
\bT_{\mu\star\nu}(u)
=\Pi\bigl(\Mt u\,\overline{\mathfrak s}_\mu,\,\overline{\mathfrak s}_\nu\bigr),
\qquad
\log\bT_{\mu\star\nu}(u)
=-\sum_{k\ge1}\frac{u^k}{k}(1-\Mq^k)(1-\Mt^{-k})
\left(\sum_{x\in\mu}[x]^{-k}\right)
\left(\sum_{y\in\nu}[y]^{-k}\right).
\end{equation}
\end{lemma}
\begin{proof}
From $[1,1]=\Mt/\Mq$, $[0,0]=1$, $[0,1]=\Mq^{-1}$, and $[1,0]=\Mt$,
\eqref{eq:Nfactor} is
\begin{equation}\label{eq:Nregular}
\mathcal N(u)=\frac{(1-u\Mq/\Mt)(1-u)}{(1-u\Mq)(1-u/\Mt)},
\end{equation}
so
\begin{align}
\log \mathcal N(u)
&=-\sum_{k\ge1}\frac{u^k}{k}
 \left(\Mq^k\Mt^{-k}+1-\Mq^k-\Mt^{-k}\right)\notag\\
&=-\sum_{k\ge1}\frac{u^k}{k}(1-\Mq^k)(1-\Mt^{-k}).
\label{eq:logN}
\end{align}
Since $[x+y]=[x][y]$, summing \eqref{eq:logN} at $u/[x+y]$ over $x\in\mu$,
$y\in\nu$ gives the second identity.  By the definition of plethystic substitution,
$p_k[\Mt u\,\overline{\mathfrak s}_\mu]=-u^k(1-\Mq^k)\sum_{x\in\mu}[x]^{-k}$
and $p_k[\overline{\mathfrak s}_\nu]=-\Mt^{-k}(1-\Mq^k)\sum_{y\in\nu}[y]^{-k}$,
so the $k$th summand in the exponent of \eqref{eq:macrkernel} is
\[
\frac{1-\Mt^k}{1-\Mq^k}\,
p_k[\Mt u\,\overline{\mathfrak s}_\mu]\,p_k[\overline{\mathfrak s}_\nu]
=-u^k(1-\Mq^k)(1-\Mt^{-k})
 \left(\sum_{x\in\mu}[x]^{-k}\right)
 \left(\sum_{y\in\nu}[y]^{-k}\right).
\]
Dividing by $k$ and summing over $k\ge1$ gives $\log\bT_{\mu\star\nu}(u)$,
which proves the first identity.
\end{proof}

Since $\epsilon_u(J_\lambda)=\Mt^{n(\lambda)}\prod_{b\in\lambda}(1-u/[b])$ and
$\log(1-u/[b])=-\sum_{k\ge1}u^k[b]^{-k}/k$, we have
\[
\kappa_{u,v}(J_\lambda)
=\Mt^{n(\lambda)}\exp\Bigl(-\sum_{k\ge1}\frac1k\,
p_k\bigl[(u+v-uv)B_\lambda(\Mq,\Mt^{-1})\bigr]\Bigr),
\]
while $p_k[uA]\,p_k[vB]=(uv)^kp_k[A]\,p_k[B]$ and Lemma~\ref{lem:TasExp} give
\[
\bT_{\mu\star\nu}(uv)
=\Pi\bigl(\Mt u\,\overline{\mathfrak s}_\mu,\,v\,\overline{\mathfrak s}_\nu\bigr).
\]
Thus the multiplicative correction factor in Theorem~\ref{thm:maintwo} is
Macdonald's Cauchy kernel, evaluated at the alphabets of $\mu$ and $\nu$
scaled by $u$ and by $v$ respectively.  The parameters enter this factor
only through their product $uv$, which is the argument of $\bT$ corresponding
to the resolvent variable in Theorem~\ref{thm:main}.

\subsection{Proof of the two-parameter identity}\label{ssec:AFLTproof}

We deduce Theorem~\ref{thm:maintwo} from a finite-rank product identity of
Warnaar \cite[Proposition~3.1]{Warnaar:2005}.

Warnaar works in $n$ variables $x=(x_1,\ldots,x_n)$, with
$\delta=(n-1,n-2,\ldots,1,0)$ and $\Mt^\delta=(\Mt^{n-1},\ldots,\Mt,1)$.  With
$(a;\Mq)_m:=\prod_{\rho=0}^{m-1}(1-a\Mq^\rho)$, for a partition
$\lambda$ with at most $n$ parts, put
\[
(a)_\lambda=(a;\Mq,\Mt)_\lambda:=\Mt^{n(\lambda)}\prod_{i=1}^{n}(a\Mt^{1-i};\Mq)_{\lambda_i},
\]
and define the (non-integral) Macdonald LR coefficients $f^\nu_{\lambda\mu}$
by
$P_\lambda(x)P_\mu(x)=\sum_\nu f^\nu_{\lambda\mu}P_\nu(x)$
\cite[(3.9)]{Warnaar:2005}.

\begin{proposition}[Warnaar, {\cite[Proposition~3.1]{Warnaar:2005}}]
\label{prop:Warnaar-finite-rank}
Let $\lambda=(\lambda_1,\ldots,\lambda_n)$ and $\mu=(\mu_1,\ldots,\mu_n)$ be
partitions.  Then
\begin{equation}\label{eq:Warnaar-finite-rank}
\sum_\nu f^\nu_{\lambda\mu}P_\nu(\Mt^\delta)\frac{(z\Mt^{-n})_\nu}{(z)_\nu}
=P_\lambda(\Mt^\delta)P_\mu(\Mt^\delta)
\prod_{i,j=1}^{n}
\frac{(z\Mt^{1-i-j};\Mq)_{\lambda_i+\mu_j}}{(z\Mt^{2-i-j};\Mq)_{\lambda_i+\mu_j}}.
\end{equation}
\end{proposition}

For given $\lambda$ and $\mu$ this is an identity of rational functions in
$z$, since $f^\nu_{\lambda\mu}=0$ unless $|\nu|=|\lambda|+|\mu|$.

\begin{corollary}\label{cor:Warnaar-rewritten}
Let $\mu,\nu$ be partitions and let $n\ge|\mu|+|\nu|$ be an integer.  Then
\begin{equation}\label{eq:Warnaar-P-finite-rank}
\kappa_{u,\Mt^n}(J_\mu J_\nu)
=\kappa_{u,\Mt^n}(J_\mu)\kappa_{u,\Mt^n}(J_\nu)
\bT_{\mu\star\nu}(u\Mt^n).
\end{equation}
\end{corollary}
\begin{proof}
Put $N=|\mu|+|\nu|$ and write
$P_\mu P_\nu=\sum_{\gamma\vdash N}f^\gamma_{\mu\nu}P_\gamma$ in $\La$.  Apply
Proposition~\ref{prop:Warnaar-finite-rank} with $(\lambda,\mu,\nu)$ renamed
$(\mu,\nu,\gamma)$ and with $z=u\Mt^n$.  Since $n\ge N$, the monomial
symmetric functions of degree $N$ remain linearly independent in $n$
variables, so the $n$-variable structure constants are the
$f^\gamma_{\mu\nu}$ of $\La$, and the sum runs over all $\gamma\vdash N$.
Here $P_\gamma(\Mt^\delta)=P_\gamma(1,\Mt,\ldots,\Mt^{n-1})=\epsilon_{\Mt^n}(P_\gamma)$,
and, by the definition of $(a;\Mq,\Mt)_\gamma$ and
Theorem~\ref{thm:principal-spec},
\[
\frac{(z\Mt^{-n};\Mq,\Mt)_\gamma}{(z;\Mq,\Mt)_\gamma}
=\prod_{i=1}^{n}\prod_{\rho=0}^{\gamma_i-1}
\frac{1-u\Mt^{1-i}\Mq^\rho}{1-u\Mt^{n+1-i}\Mq^\rho}
=\prod_{b\in\gamma}\frac{1-u/[b]}{1-u\Mt^n/[b]}
=\frac{\kappa_{u,\Mt^n}(J_\gamma)}{\epsilon_{\Mt^n}(J_\gamma)}.
\]
Since $J_\lambda=\prod_{b\in\lambda}h_\lambda(b)\,P_\lambda$,
\[
g_{\mu\nu}^{\gamma}=f_{\mu\nu}^{\gamma}\,
\frac{\prod_{b\in\mu}h_\mu(b)\prod_{b\in\nu}h_\nu(b)}{\prod_{b\in\gamma}h_\gamma(b)},
\]
and multiplying Warnaar's identity by
$\prod_{b\in\mu}h_\mu(b)\prod_{b\in\nu}h_\nu(b)$ gives
\begin{equation}\label{eq:Warnaar-specialized}
\kappa_{u,\Mt^n}(J_\mu J_\nu)
=\epsilon_{\Mt^n}(J_\mu)\epsilon_{\Mt^n}(J_\nu)
\prod_{i,j=1}^{n}
\frac{(u\Mt^{n+1-i-j};\Mq)_{\mu_i+\nu_j}}
     {(u\Mt^{n+2-i-j};\Mq)_{\mu_i+\nu_j}}.
\end{equation}
For fixed $i,j$ put $c=u\Mt^{n+1-i-j}$.  By \eqref{eq:Nregular},
\[
\mathcal N(\Mt c\Mq^{\rho+\sigma})
=\frac{(1-c\Mq^{\rho+\sigma+1})(1-\Mt c\Mq^{\rho+\sigma})}
      {(1-\Mt c\Mq^{\rho+\sigma+1})(1-c\Mq^{\rho+\sigma})},
\]
whose product over $0\le\rho<\mu_i$ and $0\le\sigma<\nu_j$ telescopes to
\[
\frac{(c\Mq^{\mu_i};\Mq)_{\nu_j}(\Mt c;\Mq)_{\nu_j}}
{(\Mt c\Mq^{\mu_i};\Mq)_{\nu_j}(c;\Mq)_{\nu_j}}.
\]
With
$(c;\Mq)_{\mu_i+\nu_j}=(c;\Mq)_{\mu_i}(c\Mq^{\mu_i};\Mq)_{\nu_j}$ this gives
\begin{equation}\label{eq:Warnaar-row-split}
\frac{(c;\Mq)_{\mu_i+\nu_j}}{(\Mt c;\Mq)_{\mu_i+\nu_j}}
=\frac{(c;\Mq)_{\mu_i}}{(\Mt c;\Mq)_{\mu_i}}
 \frac{(c;\Mq)_{\nu_j}}{(\Mt c;\Mq)_{\nu_j}}
 \prod_{\rho=0}^{\mu_i-1}\prod_{\sigma=0}^{\nu_j-1}
 \mathcal N(\Mt c\Mq^{\rho+\sigma}).
\end{equation}
Multiplying the first factor on the right of \eqref{eq:Warnaar-row-split}
over $j=1,\ldots,n$ telescopes, for fixed $i$, to
$(u\Mt^{1-i};\Mq)_{\mu_i}/(u\Mt^{n+1-i};\Mq)_{\mu_i}$, and the product over
$i$ is
\[
\prod_{b\in\mu}\frac{1-u/[b]}{1-u\Mt^n/[b]}
=\frac{\kappa_{u,\Mt^n}(J_\mu)}{\epsilon_{\Mt^n}(J_\mu)};
\]
the second factor likewise gives
$\kappa_{u,\Mt^n}(J_\nu)/\epsilon_{\Mt^n}(J_\nu)$.  Finally,
for boxes $x=(i-1,\rho)\in\mu$ and $y=(j-1,\sigma)\in\nu$,
$\Mt c\Mq^{\rho+\sigma}=u\Mt^{n+2-i-j}\Mq^{\rho+\sigma}=u\Mt^n/[x+y]$, so
the product of the remaining factors is $\bT_{\mu\star\nu}(u\Mt^n)$.
Substituting into \eqref{eq:Warnaar-specialized} gives
\eqref{eq:Warnaar-P-finite-rank}.
\end{proof}

\begin{proof}[Proof of Theorem~\ref{thm:maintwo}]
Both sides of \eqref{eq:kappa-multiplicativity} are rational functions of $v$
over $\QQ(\Mq,\Mt,u)$ whose denominators are products of factors $1-c\,uv$,
$c\in\QQ(\Mq,\Mt)^\times$, none vanishing at $v=\Mt^n$ for any integer $n$.  By
Corollary~\ref{cor:Warnaar-rewritten} they agree at the infinitely many
points $v=\Mt^n$, $n\ge|\mu|+|\nu|$, hence identically (compare the proof of
\cite[Chapter~VI, (6.17)]{Macdonald:1995}).
\end{proof}

Conversely, Theorem~\ref{thm:maintwo} implies
Proposition~\ref{prop:Warnaar-finite-rank} for every $n$.  In the notation of
Corollary~\ref{cor:Warnaar-rewritten}, put $v=\Mt^n$ and $u=z\Mt^{-n}$, and
reverse the steps of its proof.  The factorization of the finite product holds
whenever $\mu$ and $\nu$ have at most $n$ parts.  The terms with
$\ell(\gamma)>n$ drop out, because
$\epsilon_{\Mt^n}(P_\gamma)=P_\gamma(1,\Mt,\ldots,\Mt^{n-1})=0$, and for
$\ell(\gamma)\le n$ the $n$-variable structure constants agree with those of
$\La$.

\begin{proof}[Alternative proof of Theorem~\ref{thm:boxsummain}]
We recover the box-sum identity from Theorem~\ref{thm:maintwo}.  Subtract
\eqref{eq:kappa-multiplicativity} at $u=0$ from the same identity at $u$:
\begin{align}\label{eq:kappa-difference}
\sum_\lambda g_{\mu\nu}^{\lambda}
\bigl(\kappa_{u,v}(J_\lambda)-\kappa_{0,v}(J_\lambda)\bigr)
={}&\kappa_{u,v}(J_\mu)\kappa_{u,v}(J_\nu)
\bT_{\mu\star\nu}(uv)\notag\\
&-\kappa_{0,v}(J_\mu)\kappa_{0,v}(J_\nu).
\end{align}
By \eqref{eq:kappa-v1-expansion}, the left-hand side is
$(1-v)^2\varpi_\mu\varpi_\nu\sum_\lambda\widehat g_{\mu\nu}^{\lambda}\Delta_u(\lambda)
+O((1-v)^3)$.  On the right, the expansions of Section~\ref{ssec:secondparam}
give $\kappa_{u,v}(J_\mu)=\varpi_\mu(1-v)+O((1-v)^2)$ and likewise for $\nu$ and
for $\kappa_{0,v}$, while $\bT_{\mu\star\nu}(uv)=\bT_{\mu\star\nu}(u)+O(1-v)$;
so the right-hand side is
$(1-v)^2\varpi_\mu\varpi_\nu(\bT_{\mu\star\nu}(u)-1)+O((1-v)^3)$.  Comparing
the coefficients of $(1-v)^2$ proves \eqref{eq:boxsummain}.
\end{proof}

\subsection{Implications: residue identities}

We extract Littlewood--Richardson identities directly from
Theorem~\ref{thm:maintwo}. If $g_{\mu\nu}^{\lambda}\ne0$, then
$\mu\cup\nu\subseteq\lambda$ \cite[Chapter~VI, (7.4) and (8.3)]{Macdonald:1995},
and $|\lambda|=|\mu|+|\nu|$; hence
\[
|\lambda\setminus(\mu\cup\nu)|=|\mu\cap\nu|.
\]
We call the boxes of $\lambda\setminus(\mu\cup\nu)$ the \emph{extra boxes} of
$\lambda$.
For any finite set of boxes $S$, define
\begin{equation}\label{eq:PhiS}
\Phi_S(u,v)
:=\prod_{b\in S}\frac{(1-u/[b])(1-v/[b])}{1-uv/[b]}.
\end{equation}
Divide \eqref{eq:kappa-multiplicativity} by
$\kappa_{u,v}(J_\mu)\kappa_{u,v}(J_\nu)$.  Since
$\Phi_\lambda=\Phi_{\mu\cup\nu}\Phi_{\lambda\setminus(\mu\cup\nu)}$ and
$\Phi_\mu\Phi_\nu=\Phi_{\mu\cup\nu}\Phi_{\mu\cap\nu}$, this gives
\begin{equation}\label{eq:kappa-master}
\sum_\lambda g_{\mu\nu}^{\lambda}
 \Mt^{n(\lambda)-n(\mu)-n(\nu)}
\frac{\Phi_{\lambda\setminus(\mu\cup\nu)}(u,v)}
     {\Phi_{\mu\cap\nu}(u,v)}
=
\bT_{\mu\star\nu}(uv).
\end{equation}

We first record the consequence of a single diagonal residue.  Let
$y\notin\mu\cup\nu$ be a box, and take the residue of
\eqref{eq:kappa-master} along $uv=[y]$, writing $w=uv$ and keeping $u$ fixed.
Only terms for which $y$ is an extra box of $\lambda$ contribute.  The factor
indexed by $y$ has residue
\[
\operatorname*{Res}_{w=[y]}
\frac{(1-u/[y])(1-w/(u[y]))}{1-w/[y]}
=
-[y](1-u/[y])(1-u^{-1}).
\]
Setting $v=[y]/u$ in all remaining factors gives
\begin{align}\label{eq:kappa-diagonal-residue}
&\sum_{\substack{\lambda:\ y\in\lambda\setminus(\mu\cup\nu)}}
 g_{\mu\nu}^{\lambda}
 \Mt^{n(\lambda)-n(\mu)-n(\nu)}
\frac{\Phi_{\lambda\setminus(\mu\cup\nu)\setminus\{y\}}
      (u,[y]/u)}
     {\Phi_{\mu\cap\nu}(u,[y]/u)}\notag\\
&\hspace{20mm}=
-\frac{\displaystyle
 \operatorname*{Res}_{w=[y]}\bT_{\mu\star\nu}(w)}
 {[y](1-u/[y])(1-u^{-1})}.
\end{align}

If $\bT_{\mu\star\nu}(w)$ is regular at $w=[y]$, the right-hand side of
\eqref{eq:kappa-diagonal-residue} vanishes and the result is a homogeneous
relation.  At a pole of $\bT_{\mu\star\nu}$ it is an inhomogeneous relation.

\begin{corollary}[Double-residue identities]
\label{cor:kappa-double-residue}
Let $x\in\mu\cap\nu$ satisfy $x\ne(0,0)$, and let $y$ be a box such that
$x+y\notin\mu\cup\nu$.  Then
\begin{equation}\label{eq:kappa-double-residue}
\sum_{\substack{\lambda:\ x+y\in\lambda\setminus(\mu\cup\nu)}}
 g_{\mu\nu}^{\lambda}
 \Mt^{n(\lambda)-n(\mu)-n(\nu)}
\prod_{b\in\lambda\setminus(\mu\cup\nu)\setminus\{x+y\}}
\frac{(1-[x-b])(1-[y-b])}{1-[x+y-b]}
=0.
\end{equation}
\end{corollary}
\begin{proof}
Apply \eqref{eq:kappa-diagonal-residue} with its box $y$ replaced by $x+y$.
Multiply both sides by $\Phi_{\mu\cap\nu}(u,[x+y]/u)$ and set $u=[x]$.
This specialization is regular: every denominator in the remaining box
products is $1-[x+y]/[b]$ with $b\ne x+y$.  The scalar right-hand side of
\eqref{eq:kappa-diagonal-residue} is also regular at $u=[x]$, since
$[x]\ne1$ and $[x]\ne[x+y]$; here $y\ne(0,0)$ because $x\in\mu$ and
$x+y\notin\mu$.

The factor $\Phi_{\mu\cap\nu}(u,[x+y]/u)$ vanishes there because $\mu\cap\nu$
contains the box $x$.  The
right-hand side is therefore zero.  On the left, $[x+y]/[x]=[y]$, and each
remaining box factor becomes
\[
\frac{(1-[x]/[b])(1-[y]/[b])}{1-[x+y]/[b]}
=\frac{(1-[x-b])(1-[y-b])}{1-[x+y-b]}.
\]
This proves \eqref{eq:kappa-double-residue}.  Multiplying by
$\Phi_{\mu\cap\nu}$ before evaluation avoids assuming a simple pole: when $y\in\mu\cap\nu$, that
pole can be double.
\end{proof}

The excluded case $x=(0,0)$ is inhomogeneous rather than vanishing.

\begin{corollary}
\label{cor:kappa-root-residue}
Assume $\mu,\nu\ne\varnothing$, and let $y\notin\mu\cup\nu$ be a box.  Then
\begin{equation}\label{eq:kappa-root-residue}
\sum_{\substack{\lambda:\ y\in\lambda\setminus(\mu\cup\nu)}}
\widehat g_{\mu\nu}^{\lambda}
=
\frac{1}{[y]}
\operatorname*{Res}_{w=[y]}\bT_{\mu\star\nu}(w).
\end{equation}
\end{corollary}
\begin{proof}
Take the residue of \eqref{eq:kappa-diagonal-residue} at $u=1$, corresponding
to the box $(0,0)\in\mu\cap\nu$.  The residue of the right-hand side is
\[
-\frac{\operatorname*{Res}_{w=[y]}\bT_{\mu\star\nu}(w)}{[y]-1}.
\]
On the left, the root-box factor contributes the common residue $-1$.  For
every supported $\lambda$, the definition of $\varpi$ gives
\[
\frac{\varpi_\lambda}{\varpi_\mu\varpi_\nu}
=
\Mt^{n(\lambda)-n(\mu)-n(\nu)}
\frac{\displaystyle
 \prod_{b\in\lambda\setminus(\mu\cup\nu)}(1-1/[b])}
{\displaystyle
 \prod_{a\in(\mu\cap\nu)\setminus\{(0,0)\}}(1-1/[a])}.
\]
The diagonal residue has already removed the factor indexed by $y$.  Hence its
$\lambda$-summand is
\[
g_{\mu\nu}^{\lambda}
 \Mt^{n(\lambda)-n(\mu)-n(\nu)}
\frac{\displaystyle
 \prod_{b\in\lambda\setminus(\mu\cup\nu)\setminus\{y\}}(1-1/[b])}
{\displaystyle
 \prod_{a\in(\mu\cap\nu)\setminus\{(0,0)\}}(1-1/[a])}
=
\frac{\widehat g_{\mu\nu}^{\lambda}}{1-1/[y]}.
\]
Therefore the residue of the left-hand side is
\[
-\frac{1}{1-1/[y]}
\sum_{\substack{\lambda:\ y\in\lambda\setminus(\mu\cup\nu)}}
\widehat g_{\mu\nu}^{\lambda}.
\]
Equating the two residues and using
$1-1/[y]=([y]-1)/[y]$ proves \eqref{eq:kappa-root-residue}.
\end{proof}

\begin{example}[Root-box residues for $J_{(1)}^2$]
Take $\mu=\nu=(1)$, so
$\bT_{(1)\star(1)}(w)=\mathcal N(w)$.  For the box $y=(0,1)$, with $[y]=\Mq^{-1}$,
the only partition $\lambda$ with $g_{(1),(1)}^{\lambda}\ne0$ containing $y$
is $(2)$.  Hence
\[
\widehat g_{(1),(1)}^{(2)}
=\Mq\operatorname*{Res}_{w=\Mq^{-1}}\mathcal N(w)
=\frac{(1-\Mq)(1-\Mt)}{1-\Mq\Mt}.
\]
For the box $y=(1,0)$, with $[y]=\Mt$, the only such partition containing $y$
is $(1,1)$.  Thus
\[
\widehat g_{(1),(1)}^{(1,1)}
=\frac1\Mt\operatorname*{Res}_{w=\Mt}\mathcal N(w)
=-\frac{(1-\Mq)(1-\Mt)}{1-\Mq\Mt}.
\]
In particular,
\[
\widehat g_{(1),(1)}^{(2)}
+\widehat g_{(1),(1)}^{(1,1)}=0,
\]
in agreement with \eqref{eq:ghat-sum}.
\end{example}

\begin{example}[Solving for $g_{(2,1),(2,1)}^{(3,2,1)}$]\label{ex:g321}
Let $\mu=\nu=(2,1)$ and choose
\[
x=(0,1),
\qquad
y=(1,0).
\]
Both boxes belong to $\mu\cap\nu$, and neither is $(0,0)$.
Then $x+y=(1,1)$ is outside $\mu\cup\nu$.  Of the seven partitions in the
support of $J_{(2,1)}^2$, the condition
$x+y\in\lambda\setminus(\mu\cup\nu)$ selects exactly
\[
(4,2),\quad (3,3),\quad (3,2,1),\quad
(2,2,2),\quad (2,2,1,1).
\]
For brevity write $g^\lambda:=g_{(2,1),(2,1)}^\lambda$.  Substitution in
\eqref{eq:kappa-double-residue} gives
\begin{align}\label{eq:kappa-g321-relation}
0={}&
\frac{(\Mq+1)(\Mq-1)^2(\Mq^3\Mt-1)}{\Mq\Mt-1}\,g^{(4,2)}\notag\\
&-\frac{(\Mt-\Mq)(\Mq-1)(\Mq+1)(\Mq^2\Mt-1)}{\Mq\Mt-1}\,g^{(3,3)}\notag\\
&-\frac{(\Mt-1)(\Mq-1)(\Mq\Mt^2-1)(\Mq^2\Mt-1)}{(\Mq\Mt-1)^2}\,
 g^{(3,2,1)}\notag\\
&+\frac{(\Mt-\Mq)(\Mt-1)(\Mt+1)(\Mq\Mt^2-1)}{\Mq\Mt-1}\,g^{(2,2,2)}\notag\\
&+\frac{(\Mt+1)(\Mt-1)^2(\Mq\Mt^3-1)}{\Mq\Mt-1}\,g^{(2,2,1,1)}.
\end{align}
The coefficient of $g^{(3,2,1)}$ is nonzero, so
\eqref{eq:kappa-g321-relation} determines it from the other four.  We have
verified \eqref{eq:kappa-g321-relation} against the expansion of
$J_{(2,1)}^2$ in the integral Macdonald basis.
\end{example}

\section*{Acknowledgements}
The author thanks the Institut de Math\'ematiques de Jussieu--Paris Rive
Gauche and the organizers of its 2026 summer school \emph{Integrable
Combinatorics} (Sorbonne Universit\'e, Paris, 15--19 June 2026) for their
hospitality; part of this work was done during the school.

Warnaar's finite-rank product identity used in
Section~\ref{ssec:AFLTproof} was identified with assistance from GPT-5.6 Sol
(Max Thinking), which was also used to check the resulting algebraic
reductions.  Claude Opus 5.5 (Extra) was used to review and clarify several
statements and ideas.

\end{document}